\documentclass[twoside, final]{amsart}

\usepackage{amsmath, amsfonts, amsthm, amssymb}
\usepackage{exscale}
\usepackage{amscd}
\usepackage{color}

\usepackage[notref, notcite]{showkeys}

\renewcommand{\geq}{\geqslant}
\renewcommand{\leq}{\leqslant}

\numberwithin{equation}{section}
\numberwithin{equation}{subsection}

\newtheorem{proposition}{Proposition}[subsection]

\newtheorem{corollary}[proposition]{Corollary}
\newtheorem{lemma}[proposition]{Lemma}

\newtheorem{question}[proposition]{Open Question}

\newtheorem*{main-theorem}{Main Theorem}
\newtheorem*{theorem*}{Theorem}

\theoremstyle{definition}

\newtheorem{remark}[proposition]{Remark}

\newtheorem*{remark*}{Remark}

\newcommand{\Cal}{\mathcal}
\newcommand{\wcH}{\widetilde{\Cal{H}}}
\newcommand{\sk}{$\text{}$ \newline}

\def\Span {\operatorname{Span}}
\def\Spec {\operatorname{Spec}}
\def\Tr {\operatorname{Tr}}

\def\phi{\varphi}

\def\pa {\partial}

\def\ZZ{{\mathbb Z}}

\def\Rr{{\mathbb R}}

\def\CC{{\mathbb C}}

\def\Re{\,\mathrm{Re}\,}
\def\Im{\,\mathrm{Im}\,}

\def\diam{\mathrm{diam}\,}
\def\dim{\mathrm{dim}\,}

\def\SS{{\mathbb S}}

\def\phi{\varphi}

\def\dist{\operatorname{dist}}

\def\be{\begin{eqnarray*}}
\def\ee{\end{eqnarray*}}
\def\ben{\begin{eqnarray}}
\def\een{\end{eqnarray}}
\def\beq{\begin{equation}}
\def\eeq{\end{equation}}

\def\EE{\mathcal{E}}
\def\L2R{L_{\text{Rest}}^2}

\def\11{\mathds{1}}

\def\HH{\mathbb{H}}

\def\RR{\mathbb{R}}

\def\L2c{L^2_{\text{comp}}}

\def\calH{\mathcal{H}}

\def\II{\mathcal{I}}
\def\EE{\mathcal{E}}

\def\11{\mathbb{1}}
\def\Vol{\text{Vol}}

\begin{document}

\title[Vortices]{Higher Dimensional Vortex Standing Waves for Nonlinear
Schr{\"o}dinger Equations}

\author[J.L. Marzuola]{Jeremy L. Marzuola}
\email{marzuola@math.unc.edu}
\address{Department of Mathematics, UNC-Chapel Hill \\ CB\#3250
  Phillips Hall \\ Chapel Hill, NC 27599}

\author[M. Taylor]{Michael E. Taylor}
\email{met@math.unc.edu}
\address{Department of Mathematics, UNC-Chapel Hill \\ CB\#3250
  Phillips Hall \\ Chapel Hill, NC 27599}


\begin{abstract}
We study standing wave solutions to nonlinear Schr{\"o}dinger equations,
on a manifold with a rotational symmetry, which
transform in a natural fashion under
the group of rotations. We call these vortex
solutions.  They are higher dimensional versions of vortex standing waves
that have been studied on the Euclidean plane.  We focus on two types of
vortex solutions, which we call spherical vortices and axial vortices.
\end{abstract}

\maketitle

\section{Introduction}
\label{sec:vortintro}

A standing wave solution to the nonlinear Schr\"odinger equation
\begin{equation}
\label{eqn:nls}
i v_t + \Delta v + |v|^{p-1} v = 0
\end{equation}
is a solution of the form
\begin{equation}
\label{eqn:soliton}
v(t,x) = e^{i \lambda t} u(x),
\end{equation}
where $u$ solves the nonlinear elliptic partial differential equation
\begin{equation}
\label{eqn:nlbs}
-\Delta u + \lambda u - |u|^{p-1} u = 0.
\end{equation}
Here $u$ is defined on a complete Riemannian manifold $M$
(possibly with boundary),
and $\Delta$ is the Laplace-Beltrami operator on $M$.  If $\pa M \neq 0$,
we might impose
Dirichlet or Neumann boundary conditions.  Similarly, a standing wave
solution to the nonlinear Klein-Gordon equation
\begin{equation}
\label{eqn:nlkg}
v_{tt} - \Delta v + \sigma^2 v - |v|^{p-1} v = 0, \ \ v(t,x) = e^{i \mu t} u (x),
\end{equation}
leads to \eqref{eqn:nlbs} with $\lambda = \sigma^2 - \mu^2$. 

There is a large literature on \eqref{eqn:nlbs} when $M$ is a Euclidean space
$\RR^n$ or a bounded domain.  More recent papers have dealt with hyperbolic
space $\HH^n$ and ``weakly homogeneous spaces.''
See \cite{ManSan,ChMa-hyp,CMMT}.  One way to get solutions to \eqref{eqn:nlbs}
is to minimize the functional
\begin{equation}
\label{eqn:Flambda}
F_\lambda (u) = \| \nabla u \|_{L^2}^2 + \lambda \| u \|_{L^2}^2,
\end{equation}
subject to the constraint
\begin{equation}
\label{eqn:Jp}
J_p (u) = \int_M | u|^{p+1} \,d \Vol = \beta,
\end{equation}
with $\beta \in (0,\infty)$ fixed.
This works for the spaces $M$ mentioned above if $n \geq 2$ and
\begin{equation}
\label{eqn:pbounds}
1 < p < \frac{n+2}{n-2},
\end{equation}
provided
\begin{equation}
\label{eqn:spec}
\Spec (-\Delta) \subset [\delta, \infty), \ \ \lambda > - \delta.
\end{equation}
However, as pointed out in \cite{CMMT}, there are many examples of innocent
looking $M$, such as
\begin{equation}
\label{eqn:extdisc}
M = \RR^n \setminus B,
\end{equation}
where $B \subset \RR^n$ is a smoothly bounded compact domain, for which such
$F_\lambda$ minimizers (with Dirichlet boundary conditions) do not exist.

On the other hand, we can sometimes find $F_\lambda$-minimizers on manifolds
with some symmetry, if we add an extra constraint.  For example, suppose $M$
is as in \eqref{eqn:extdisc}, and
\begin{equation}
\label{eqn:disc}
B = \left\{  x \in \RR^n \ : \ |x| \leq 1 \right\}.
\end{equation}
Then, we can minimize $F_\lambda$ over
\begin{equation}
\label{eqn:H1rad}
H^1_r (M) = \{ u \in H^1 (M) \ : \ u \ \text{is radial} \},
\end{equation}
subject to the constraint \eqref{eqn:Jp}, as long as
\eqref{eqn:pbounds}--\eqref{eqn:spec} hold (by an argument we will generalize
in Section \ref{sec:vortspher}).  The resulting minimizer will then be a
radial solution to \eqref{eqn:nlbs}.

Here is a variant.  Take $n=2$ and let
$M = \RR^2$ or $M = \RR^2 \setminus B$, with $B$ as in \eqref{eqn:disc}.
Take $\ell \in \ZZ$ and consider
\begin{equation}
\label{eqn:H1ell}
H^1_{(\ell)} (M) = \{ u \in H^1 (M) \ : \ u( R_\theta x) =
e^{ i \ell \theta} u(x), \forall \ \theta \in [0, 2 \pi] \},
\end{equation}
where $R_\theta : \RR^2 \to \RR^2$ is a rotation:
\begin{equation}
\label{eqn:rot}
R_\theta = \left(  \begin{array}{cc}
\cos \theta & - \sin \theta \\
\sin \theta & \cos \theta
\end{array}
\right).
\end{equation}
If $\ell = 0$, we get \eqref{eqn:H1rad}, and other values of $\ell$ yield
other spaces.  We can minimize $F_\lambda$ over $H^1_{(\ell)} (M)$, subject
to the constraint \eqref{eqn:Jp}, as long as
\eqref{eqn:pbounds}--\eqref{eqn:spec} hold, and get a solution in
$H^1_{(\ell)} (M)$ to \eqref{eqn:nlbs}.  Such a solution is called a
{\it vortex solution}.  Works on this include \cite{FG}, \cite{M}, and
\cite{SZ}.

Here, we extend the scope of the search for solutions to \eqref{eqn:nlbs}
with symmetry, in several ways.  First, we let $V$ be a $k$-dimensional
complex inner-product space, and consider functions $u$ with values in $V$.
We say $u \in H^s (M,V)$ if the components of $u$ belong to the Sobolev space
$H^s (M)$.  We then set
\begin{equation}
\label{eqn:Hspi}
H^s_\pi (M) = \{ u \in H^s (M,V) \ : \ u(gx) = \pi (g) u(x), \ \forall x \in
M, \ g \in G \},
\end{equation}
where $G$ is a compact Lie group that acts on $M$ by isometries and $\pi$ is
a unitary representation of $G$ on $V$.

In case $M$ is
given by \eqref{eqn:extdisc}--\eqref{eqn:disc}, we could take $G = SO(n)$.
More generally, $G$ can be any compact subgroup of $SO(n)$ that acts
transitively on the unit sphere $\SS^{n-1} \subset \RR^n$.  Examples
include
\beq
\alignedat 2
G & = SO(n), &{} \\
G & = SU(m) \quad &(n = 2m), \\
G & = U(m) \quad &(n = 2m),  \\
G  & = Sp(m) \quad &(n = 4m) .
\label{eqn:Gclasses}
\endalignedat
\eeq
We will also consider more general complete, $n$-dimensional, Riemannian
manifolds on which such groups act in Section \ref{sec:vortspher}.  The
manifolds we treat have the form
\begin{equation}
\label{eqn:sphereM}
M = I \times \SS^{n-1}, \ \ I = [0,\infty), \ \ [1, \infty), \ \
(-\infty, \infty),
\end{equation}
where, in the first case, all points $\{ (0,\omega) \ : \ \omega \in
\SS^{n-1} \}$ are identified.  The metric tensor will have the form
\begin{equation}
\label{eqn:metric}
g = dr^2 + h (r) 
\end{equation}
with $r \in I$ and $h(r)$ an $r$-dependent family of metric tensors on
$\SS^{n-1}$, invariant under the action of $G$.  See Section
\ref{sec:vortspher} for details.  We say such manifolds have
{\it rotational symmetry}, and we say that elements of $H^1_\pi (M)$ are
{\it spherical vortices} (and solutions to \eqref{eqn:nlbs} with $u \in
H^1_\pi (M)$ are {\it spherical} vortex standing waves) in this setting.

In Section \ref{sec:vortaxial}, we consider further variants of
\eqref{eqn:H1ell}.  Here, we take $M = \RR^{n+k}$, with $G = SO(n)$
acting on the first $n$ coordinates, and acting transitively on $\SS^{n-1}$,
as in \eqref{eqn:Gclasses}.  Again, we define $H^1_\pi (M)$ as in
\eqref{eqn:Hspi}.  One example is
\beq
\aligned
&M= \RR^3, \ G = SO(2), \\
&H^1_{(\ell)}  (\RR^3) = \{ u \in H^1 (\RR^3) \ : \ u( R_\theta x) =
e^{ i \ell \theta} u(x), \ \forall \theta \in [0, 2 \pi] \},
\endaligned
\eeq
where, in place of \eqref{eqn:rot}, we have
\begin{equation}
\label{eqn:rot3d}
R_\theta = \left(  \begin{array}{ccc}
\cos \theta & - \sin \theta & 0 \\
\sin \theta & \cos \theta & 0 \\
0 & 0 & 1
\end{array}
\right).
\end{equation}
In this setting, we say that the elements of $H^1_\pi ( \RR^{n+k})$ are
{\it axial} vortices.

In the settings of both Section \ref{sec:vortspher} and Section
\ref{sec:vortaxial}, in order to have nontrivial solutions of
\eqref{eqn:nlbs} in $H^1_\pi (M)$, we need to have $H^1_\pi (M) \neq 0$.
We have the following criterion, valid when $M$ is as considered in either
of these settings.  Pick a point $p_0 \in \SS^{n-1}$, and let
\begin{equation}
\label{eqn:Krep1}
K = \{ g \in G \ : \ g \cdot p_0 = p_0 \}.  
\end{equation}
Then, $H^1_\pi (M) \neq 0$ if and only if
\begin{equation}
\label{eqn:V0}
V_0 = \{ \phi \in V \ : \ \pi (k) \phi = \phi, \ \forall k \in K \}
\end{equation}
has the property
\begin{equation}
\label{eqn:V0neq0}
V_0 \neq 0.
\end{equation}
When \eqref{eqn:V0neq0} holds, we have elements of $H^1_\pi (M)$ of the form
\begin{equation}
\label{eqn:rgp}
u(r, g \cdot p_0) = \psi (r) \pi (g) \phi, \quad
\psi : I \to \CC, \ \phi \in V_0,
\end{equation}
given $M$ as is \eqref{eqn:sphereM}--\eqref{eqn:metric}.  If $M = \RR^{n+k} =
\RR^n \times \RR^k$, we can take
\begin{equation}
\label{eqn:psipiphi}
u(r g \cdot p_0, y) = \psi(r,y) \pi (g) \phi, \quad \psi : [0,\infty) \times
\RR^k \to \CC, \ \phi \in V_0.
\end{equation}
We complement \eqref{eqn:Gclasses} with the list of pairs $(G,K)$:
\beq
\alignedat 2
G &= SO (n),\quad &K = SO (n-1), \\
G &= SU (m),\quad &K = SU (m-1), \\
G &= U (m),\quad &K = U (m-1). 
\endalignedat
\label{eqn:GKclasses}
\eeq

Here are some examples of representations of $SO(n)$ that satisfy
\eqref{eqn:V0neq0}.  Obviously the ``standard'' representation of $SO(n)$ on
$\RR^n$ (complexified to $V = \CC^n$) enjoys the property \eqref{eqn:V0neq0}.
Others arise as follows.  The action of $SO(n)$ on $\SS^{n-1}$ gives a
unitary action of $SO(n)$ on $L^2 (\SS^{n-1})$, which commutes with the
Laplacian on $\SS^{n-1}$.  Hence, we get a unitary representation of $SO(n)$
on each eigenspace $E_\ell$ of this Laplacian.  Each such eigenspace contains
a zonal function $\phi$, unique up to a multiple, and this satisfies
\eqref{eqn:V0neq0}.  This result has a converse.  If $\pi$ is an irreducible
unitary representation of $SO(n)$ on $V$ and \eqref{eqn:V0neq0} holds,
$\pi$ is equivalent to a representation on some $E_\ell$ described above, via
\begin{equation}
\label{eqn:Phi}
\Phi : V \to C^\infty ( \SS^{n-1}) , \quad
\Phi (\psi) ( g \cdot p_0) = (\pi (g)
\phi, \psi )_{L^2},
\end{equation}
with $\phi \in V_0$.  For $n = 2$, the eigenspaces $E_\ell$ break up into
$\Span (e^{ i \ell \theta})$ and $\Span (e^{- i \ell \theta})$, and we get
the spaces \eqref{eqn:H1ell}.  All the irreducible unitary representations
of $SO(3)$ arise from decomposing $L^2 (\SS^2)$ as described above.  For
$n \geq 4$ there are irreducible unitary representations of $SO(n)$ that do
not arise from such a decomposition of $L^2 (\SS^{n-1})$, and for such
representations \eqref{eqn:V0neq0} fails.

See Appendix \ref{app:H1pineq0} for further material on the validity of
\eqref{eqn:V0neq0}, applicable to more general pairs $(G,K)$, including the
second and third cases in \eqref{eqn:GKclasses}.

The rest of this paper is structured as follows.  Section \ref{sec:vortspher}
treats spherical vortices, with the goal of constructing spherical vortex
standing wave solutions to \eqref{eqn:nlbs}.  We treat $F_\lambda$-minimizers
in Section \ref{sec:vortspherFlambda}, for a class of
$n$-dimensional Riemannian manifolds with rotational symmetry.
In Section \ref{sec:vortspherenmin} we treat energy minimizers, i.e.,
minimizers of the energy
\begin{equation}
\label{eqn:energy}
E(u) = \frac12 \|
\nabla u \|_{L^2 (M) }^2 - \frac{1}{p+1} \int_M | u |^{p+1} d\Vol
\end{equation}
over $H^1_\pi (M)$, subject to the constraint
\begin{equation}
\label{eqn:charge}
Q(u) = \| u \|_{L^2}^2 = \beta
\end{equation}
with $\beta$ given in $(0,\infty)$.  In these sections, $\lambda$ and $p$
satisfy appropriate hypotheses, which are stricter for $p$ in Section
\ref{sec:vortspherenmin} than in Section \ref{sec:vortspherFlambda}.
In Section \ref{sec:vortspherWmax}, we treat maximizers of the Weinstein
functional
\begin{equation}
\label{eqn:W}
W(u)=\frac{\|u\|_{L^{p+1}}^{p+1}}{\| u\|_{L^2}^\alpha\|\nabla u\|_{L^2}^\beta}, 
\end{equation}
over nonzero $u \in H^1_\pi (\RR^n)$, with
\begin{equation}
\label{eqn:alphabeta}
\alpha = 2 - \frac{(n-2)(p-1)}{2},\quad \beta = \frac{ n (p-1)}{2},
\end{equation}
assuming $p$ satisfies \eqref{eqn:pbounds}.  Section \ref{sec:vortspherODE}
derives ordinary differential equations associated to spherical vortices,
and uses these to obtain further results about the behavior of such solutions.

Section \ref{sec:vortaxial} treats axial vortices on $\RR^{n+k}$.  We 
construct $F_\lambda$-minimizers, under appropriate hypotheses on $p$,
in Section \ref{sec:vortaxialFlambda}.  Variants of this analysis, coupled with
arguments from Sections \ref{sec:vortspherenmin} and \ref{sec:vortspherWmax},
can be brought to bear to produce axial vortices that either minimize
energy or maximize the Weinstein functional, within appropriate function
classes, though we do not pursue the details here.  
In rough parallel to the ODE study done in Section \ref{sec:vortspherODE},
Section \ref{sec:vortaxialPDE} derives reduced variable PDE for axial standing
waves.  Harnack estimates on solutions to these reduced PDE provide some
valuable information on axial vortices.

In Section \ref{sec:masscrit}, we discuss solutions to the mass critical NLS
with vortex initial data.  Thus we take $p=1+4/n$ in \eqref{eqn:nls}.
We show in \S{\ref{s4.1}} that if such initial
data $u_0 \in H^1_\pi (\RR^{n})$ has mass $\| u_0 \|_{L^2}$ less than that
of the corresponding Weinstein functional maximizer (call it $Q_\pi$),
then the solution to \eqref{eqn:nls}
exists for all $t$, extending previous results established in \cite{W} for
$u \in H^1 (\RR^n)$ with mass less than the radial Weinstein functional
maximizer and in \cite{FG} for $u_0 \in H^1_{(\ell)} (\RR^2)$.
We also treat some more general
Riemannian manifolds with rotational symmetry.
In \S{\ref{s4.2}}, we obtain some ``monotonicity '' results, on how the
Weinstein functional supremum depends on the representation $\pi$.
In Section \ref{sec:vortscat}, we address the scattering of solutions with
mass below the vortex mass.  We adapt results presented in  \cite{Caz_book}
to show that if
\beq
v_0\in H^1_\pi(\Rr^n)\cap H^{0,1}(\Rr^n),
\label{1.0.31}
\eeq
where $H^{0,1}(\Rr^n)=\{v_0\in L^2(\Rr^n):|x|v_0\in L^2(\Rr^n)\}$, and if
\beq
\|v_0\|_{L^2}<\|Q_\pi\|_{L^2},
\label{1.0.32}
\eeq
then the global solution to \eqref{eqn:nls}, with initial data $v_0$,
guaranteed by Section \ref{s4.1}, exhibits scattering.

We end with some appendices. The first goes further into which representations
$\pi$ of $G$ on $V$ satisfy \eqref{eqn:V0neq0}.
The second discusses important irreducible
representations of $G = SO(n)$, $SU(n/2)$, and $U(n/2)$, when $n = 4$,
as they relate to the description of $V_0$.
The third introduces a more general geometrical setting, unifying the treatment of
axial vortices in Section \ref{sec:vortaxial} with work done on ``weakly
homogeneous spaces'' in \cite{CMMT}.

$\text{}$ \newline
{\sc Remark.} When $\pa M\neq\emptyset$ and we want to use the Dirichlet
boundary condition, of course we replace the $s=1$ case of \eqref{eqn:Hspi}
by
\beq
H^1_\pi(M)=\{u\in H^1_0(M,V):u(gx)=\pi(g)u(x),\ \forall\, x\in M, g\in G\}.
\eeq

\section{Spherical Vortices}
\label{sec:vortspher}

In this section, we work with the following class of complete Riemannian
manifolds $M$ (possibly with boundary).  Let $n = \dim M$.  Assume $n \geq 2$.
We assume $M$ has a compact group $G$ of isometries whose orbits foliate $M$
into hypersurfaces diffeomorphic to $\SS^{n-1}$ (plus perhaps one point $o$,
fixed by the action of $G$).  Also, assume $M$ is connected, and noncompact.
Such a Riemannian manifold will be said to have {\it rotational symmetry}.
Topologically, $M$ takes one of the following forms:
\begin{align}
\label{eqn:M1}
M =\ & [0,\infty) \times \SS^{n-1} / \sim, \\
\label{eqn:M2}
& [1,\infty) \times \SS^{n-1}, \\
\label{eqn:M3}
& (-\infty, \infty) \times \SS^{n-1}.
\end{align}
For short, we say
\begin{equation}
\label{eqn:Mshort}
M = I \times \SS^{n-1}, \ \ I = [0,\infty), \ [1, \infty), \ \text{or} \
(-\infty, \infty).
\end{equation}
In the first case, all points $\{ (0,\omega) \ : \ \omega \in \SS^{n-1} \}$
are identified with $o$.  In the second case, $M$ has a boundary,
diffeomorphic to $\SS^{n-1}$.  In the third case, $M$ has neither a boundary
nor a point fixed by $G$.

Examples of spaces with rotational symmetry of the first type include $\RR^n$
and $\HH^n$, and all other noncompact rank-one symmetric spaces, as well as
vastly more cases.  Examples of the second type can be obtained by excising
a ball centered at $o$ from examples of the first type.  Examples of the
third type can be obtained by gluing together two examples of the second type.

The metric tensor on $M$ takes the form
\begin{equation}
\label{eqn:metric1}
g = dr^2 + h(r),
\end{equation}
where $r$ parametrizes the interval $I$ and $h(r)$ is an $r$-dependent family
of metric tensors on $\SS^{n-1}$, invariant under the action of $G$.  Since
$G$ acts transitively on $\SS^{n-1}$, the area element on $\SS^{n-1}$ induced
by $h(r)$ is an $r$-dependent multiple of the standard area element $d \SS$
on $\SS^{n-1} \subset \RR^n$, and the volume element on $M$ has the form
\begin{equation}
\label{eqn:volel1}
d \Vol = A(r)\, dr\, d \SS (\omega), \ \omega \in \SS^{n-1}.
\end{equation}
As is well-known, 
\begin{align}
\label{eqn:eucjac}
& M = \RR^n \Rightarrow A(r) = A_n r^{n-1}, \\
\label{eqn:hypjac}
& M= \HH^n \Rightarrow A(r) = A_n (\sinh r)^{n-1}.
\end{align}
In Section \ref{sec:vortspherFlambda}, we show that $F_\lambda$-minimizers
exist in $H^1_\pi (M)$ on manifolds with rotational symmetry, under appropriate
hypotheses.  Recall from Section \ref{sec:vortintro} that
\begin{equation}
\label{eqn:Flambda1}
F_\lambda (u) = \| \nabla u \|_{L^2}^2 + \lambda \| u \|_{L^2}^2.
\end{equation}
We fix $\beta \in (0,\infty)$ and desire to minimize $F_\lambda (u)$ over
$u \in H^1_\pi (M)$, subject to the constraint
\begin{equation}
\label{eqn:Jp1}
J_p = \int_M |u|^{p+1}\, d \Vol = \beta.
\end{equation}
We assume
\begin{equation}
\label{eqn:spec1}
\Spec (-\Delta) \subset [\delta, \infty), \ \lambda > - \delta,
\end{equation}
and
\beq
\aligned
1 < p < \frac{n+2}{n-2}  \quad &\text{if} \ I = [0,\infty), \\
1 < p < \infty  \quad &\text{if} \ I = [1,\infty) \ \text{or} \
(-\infty,\infty).
\endaligned
\label{eqn:pconstraints}
\eeq
We will also impose a certain condition on the function $A(r)$ arising in
\eqref{eqn:volel1}.  See Lemma \ref{lem:intvolel}.  This result guarantees,
among other things, that
\begin{equation}
\label{eqn:sobemb}
H^1_\pi (M) \subset
L^{p+1} (M), \quad \forall \ p \in \left(  1, \frac{n+2}{n-2} \right).
\end{equation}
If $u$ and $\psi$ belong to $H^1_\pi (M)$, standard calculations yield
\beq
\aligned
\label{eqn:dFlambda}
\frac{d}{d \tau} F_\lambda (u + \tau \psi) \Bigr|_{\tau = 0}  & =
2 \Re ( - \Delta u + \lambda u, \psi), \\
\frac{d}{d \tau} J_p (u + \tau \psi) \Bigr|_{\tau = 0} & = (p+1)
\Re \int |u|^{p-1} (u, \psi)_{V}\, d \Vol.
\endaligned
\eeq
If $u \in H^1_\pi (M)$ is an $F_\lambda$- minimizer, subject to the constraint
\eqref{eqn:Jp1}, then
\beq
\aligned
\label{eqn:EL}
& \psi \in H^1_\pi (M) \ \text{and} \ \Re \int |u|^{p-1} (u, \psi)_V\,
d \Vol = 0 \\
& \Longrightarrow \Re (-\Delta u + \lambda u , \psi ) = 0. 
\endaligned
\eeq
Note that from \eqref{eqn:sobemb}, we have
\begin{equation}
\label{eqn:udual}
u \in H^1_\pi (M) \Longrightarrow |u|^{p-1} u \in H^{-1}_\pi (M).
\end{equation}
The space $H^{-1}_\pi (M)$ is canonically dual to $H^1_\pi (M)$.
(This requires slightly more elaborate wording if $\pa M \neq 0$,
but the appropriate duality theory is standard.)  It follows from
\eqref{eqn:EL} that the two elements of $H^{-1} (M)$ given by $-\Delta u +
\lambda u$ and $|u|^{p-1} u$ are linearly dependent over $\RR$.  Hence,
there exists a $K \in \RR$ such that
\begin{equation}
\label{eqn:lagmulnlbs}
-\Delta u + \lambda u = K |u|^{p-1} u,
\end{equation}
with equality holding in $H^{-1}_\pi (M)$.  Pairing both sides of
\eqref{eqn:lagmulnlbs} with $u$ gives
\begin{equation}
\label{eqn:Kdef}
K = \beta^{-1} \II (\beta, \pi) > 0
\end{equation}
where
\begin{equation}
\label{eqn:calI}
\II (\beta , \pi) = \inf \{ F_\lambda (u) \ : \ u \in H^1_\pi (M), \ J_p (u)
= \beta \}.
\end{equation}
If $u$ solves \eqref{eqn:calI}, then $u_a = a u$ solves
\begin{equation}
\label{eqn:ua}
-\Delta u_a + \lambda u_a = |a|^{-(p-1)} K |u_a|^{p-1} u_a,
\end{equation}
so taking $a = K^{{1}/({p-1})}$ yields a solution in $H^1_\pi (M)$ to
\eqref{eqn:nlbs}.

In Section \ref{sec:vortspherenmin}, we discuss the existence of energy
minimizers in $H^1_\pi (M)$.  Recall from Section \ref{sec:vortintro} that
\begin{equation}
\label{eqn:energy1}
E(u) = \frac12 \| \nabla u \|_{L^2}^2 - \frac{1}{p+1} \int_M | u |^{p+1}\, d
\Vol.
\end{equation}
We fix $\beta \in (1, \infty)$ and desire to minimize $E(u)$ over $H^1_\pi
(M)$, subject to the constraint
\begin{equation}
\label{eqn:massbeta}
Q(u) = \| u \|_{L^2}^2 = \beta.
\end{equation}
In place of \eqref{eqn:pconstraints}, we assume
\beq
\aligned
\label{eqn:pconstraints_enmin}
1 < p < 1 + \frac{4}{n}, \quad &\text{if} \ I = [0,\infty), \\
1 < p < \infty, \quad &\text{if} \ I = [1,\infty) \ \text{or} \ (-\infty, \infty).
\endaligned
\eeq
We also impose conditions on $A(r)$ as in Lemma \ref{lem:intvolel}, and
further conditions satisfied in Section \ref{sec:vortspherenmin}.  The
argument here is less elementary than that of Section
\ref{sec:vortspherFlambda}; it involves the concentration-compactness method
of \cite{L1}.  See Section \ref{sec:vortspherenmin} for further details.

In Section \ref{sec:vortspherWmax} we obtain Weinstein functional maximizers.
Here $M = \RR^n$ and $W(u)$ is as in \eqref{eqn:W}.  We obtain a maximizer
$u \in H^1_\pi (\RR^n)$ under hypotheses as in the setting of finding
$F_\lambda$-minimizers.

Section \ref{sec:vortspherODE} derives ordinary differential equations
associated to vortex standing waves obtained in Sections
\ref{sec:vortspherFlambda}--\ref{sec:vortspherWmax}.  Results here also lead
to ``positivity'' results for the standing waves, restricted to a ray in $M$,
upon multiplying by a constant, at least under certain conditions on the
space $V_0$.

\subsection{$F_\lambda$-minimizers}
\label{sec:vortspherFlambda}

We take $M$ to be a complete, $n$-dimensional, Riemannian manifold (possibly
with boundary), with rotational symmetry, as defined above.  We make the
hypotheses \eqref{eqn:spec1}--\eqref{eqn:pconstraints} and define $F_\lambda
(u)$ and $J_p (u)$ as in \eqref{eqn:Flambda1}-\eqref{eqn:Jp1}.  Hypothesis
\eqref{eqn:spec1} implies
\begin{equation}
\label{eqn:FlamH1}
F_\lambda (u) \approx  \| u \|_{H^1 (M)}^2.
\end{equation}
The Sobolev embedding theorem together with the Rellich theorem implies
\begin{equation}
\label{eqn:compemb}
H^1 (M) \hookrightarrow  L^q (\Omega), \ \text{compactly}, \quad
\forall \ q \in \left[ 2, \frac{2n}{n-2} \right),
\end{equation}
given $\Omega \subset M$ relatively compact.  If $M$ is as in
\eqref{eqn:Mshort}-\eqref{eqn:metric1} and $I$ is $[1,\infty)$ or $(-\infty,
\infty)$, then, just as in the familiar case of radial functions,
\begin{equation}
\label{eqn:compemb1}
H^1_\pi (M) \hookrightarrow L^q (\Omega), 
\text{compactly}, \quad \forall \ q \in \left[ 2, \infty \right),
\end{equation}
for such $\Omega$.  Note that if $p$ satisfies \eqref{eqn:pconstraints}, then
$q = p+1$ satisfies \eqref{eqn:compemb} or \eqref{eqn:compemb1}.  We fix
$\beta$ and the representations $\pi$ of $G$ on $V$, and set
\begin{equation}
\label{eqn:Ical}
\II ( \beta, \pi) = \inf \{ F_\lambda (u) \ : \ u \in H^1_\pi (M), \ J_p (u)
= \beta \}.
\end{equation}
In order to know that $\II (\beta, \pi) > 0$, we need to supplement
\eqref{eqn:compemb} with a global estimate in $L^q (M)$.  We need this, not
for all $u \in H^1 (M,V)$, just for all $u \in H^1_\pi (M)$.  The following
observation will prove useful.

Given $u \in H^1 (M,V)$, set $w = |u|$, using the $V$-norm.  There is the
classical (elementary) inequality:
\begin{equation}
\label{eqn:kinetic}
w = |u| \Longrightarrow  | \nabla w | \leq | \nabla u|,
\end{equation}
valid for any inner product space $V$.  Now,
\begin{equation}
\label{eqn:kineticpi}
u \in H^1_\pi (M) \Longrightarrow w = |u| \in H^1_r (M), \ \| w \|_{H^1} \leq
\| u \|_{H^1},
\end{equation}
where $H^1_r (M)$ consists of radial functions on $M$ (the case $H^1_\pi$
where $\pi$ is the trivial representation).  In light of this, the following
is the key to success.  Recalling \eqref{eqn:M1}--\eqref{eqn:volel1}, let us
set
\begin{equation}
\label{eqn:MR}
M_R = \{ x \in M \ : \ |r | \geq R \}.
\end{equation}
  
\begin{lemma}
\label{lem:intvolel}
Assume that $A(r)$ in \eqref{eqn:volel1} satisfies either
\begin{equation}
\label{eqn:intvolel}
\int_{|r| \geq 1} \frac{dr}{A(r)} < \infty,
\end{equation}
or
\begin{equation}
\label{eqn:Aasymp}
\lim_{|r| \to \infty} A(r) = \infty, \ \text{and} \ \sup_{|r| \geq 1} \left|
\frac{A' (r)}{A(r)} \right| < \infty.
\end{equation}
Then,
\beq
\aligned
\label{eqn:wlim}
w \in H^1_r (M) \Longrightarrow\ &w |_{M_1} \in C( M_1), \ \text{and}  \\
&\lim_{|r| \to \infty} |w (r) | = 0.
\endaligned
\eeq
\end{lemma}

\begin{remark}
By \eqref{eqn:eucjac} and \eqref{eqn:hypjac}, \eqref{eqn:intvolel} holds for
$\RR^n$ whenever $n \geq 3$ and for $\HH^n$ whenever $n \geq 2$, and
\eqref{eqn:Aasymp} holds for both $\RR^n$ and $\HH^n$ whenever $n \geq 2$.
The implication \eqref{eqn:wlim} was proven for $M = \RR^n$ in \cite{Str}
(see also \cite{BeLi}) and it was proven for $M = \HH^n$ in \cite{CMMT}.
The proof here for Lemma \ref{lem:intvolel} is a variation of these arguments.
\end{remark}

\sk
To prove Lemma \ref{lem:intvolel}, it suffices to establish
\eqref{eqn:wlim} assuming $w$ is smooth, non-negative,
and that $w = 0$ for $|r| \leq 1$.  For simplicity, we assume $r$ runs over
$[0,\infty)$ as in \eqref{eqn:M1}.

Set $B(r) = (A(r))^{1/2}$.  Then,
\begin{align}
\frac{d}{dr} \bigl(  A(r) w^2 \bigr)
& = 2 \frac{d}{dr} \Bigl( B(r) w \Bigr)
B(r) w \notag \\
& \leq \left[ \frac{d}{dr}  \left( B(r) w \right) \right]^2 + \left( B(r) w
\right)^2 \label{eqn:Avirial} \\
& = \left[ B(r) \frac{dw}{dr} + B' (r) w \right]^2 + \left( B(r) w \right)^2
\notag \\
 & \leq 2 A(r) \left( \frac{dw}{dr} \right)^2 + 2 B' (r)^2 w^2 + A(r) w^2,
 \notag
 \end{align}
the last inequality by $(a+b)^2 \leq 2 a^2 + 2 b^2$.  Now,
\beq
\aligned
\label{eqn:Bprime}
B(r) = A(r)^{1/2}
& \Rightarrow B' (r) = \frac12 A(r)^{-1/2} A'(r) \\
& \Rightarrow 2 B'(r)^2 = \frac12  A(r)^{-1} A' (r)^2,
\endaligned
\eeq
  so
  \begin{equation}
  \label{eqn:Avirial1}
  \frac{d}{dr} \Bigl( A(r) w^2 \Bigr) \leq 2 A(r) \left( \frac{dw}{dr}
  \right)^2 + A(r) \left[ 1 + \left( \frac{ A' (r) }{ A(r)} \right)^2
  \right] w^2.
  \end{equation}
  Hence, given $w(1) = 0$, we have for $r > 1$
  \begin{equation}
  \label{eqn:Avirial2}
  A(r) w(r)^2 \leq 2 \int_{M_1} | \nabla w |^2 \Vol + \int_{M_1}
  \left[ 1 + \left( \frac{ A' (r) }{ A(r)} \right)^2 \right] |w|^2 d \Vol.
  \end{equation}
  This shows that \eqref{eqn:Aasymp} implies \eqref{eqn:wlim}.

  Whether or not \eqref{eqn:intvolel} and \eqref{eqn:Aasymp} hold, $A(r)$ is
  smooth and positive for $r \neq 0$, so
  $w \in H^1_r (M) \Rightarrow w |_{M_1}
  \in C(M_1)$.  We now seek to prove the estimate \eqref{eqn:wlim} under the
  hypothesis \eqref{eqn:intvolel}.  Let us assume $R > 2$ and set
  \begin{equation}
  \label{eqn:wR}
  w_R (r) = \chi_R (r) w(r),
  \end{equation}
  where $\chi_R (r) = 0$ for $|r| \leq R-1$, $1$ for $|r| \geq R$, and $\chi_R$
  has Lipschitz constant $1$.  Then,
\beq
\aligned
\| w \|_{L^\infty (M_R)} & \leq \int_{R-1}^\infty |w_R' (r) | dr,  \\
& = \int_{R-1}^\infty | w_R' (r)| A(r)^{1/2} A(r)^{-1/2} dr
\label{eqn:MRbound} \\
& \leq \eta (R) \| w \|_{H^1 (M)},
\endaligned
\eeq
by Cauchy's inequality, where
\begin{equation}
\label{eqn:etaAbd1}
\eta (R) = \left( \int_{R-1}^\infty \frac{dr}{A(r)} \right)^{1/2}.
\end{equation}
Hypothesis \eqref{eqn:intvolel} implies $\eta (R) \to 0$ as $R \to \infty$,
again yielding \eqref{eqn:wlim}.
This proves Lemma \ref{lem:intvolel}.

$\text{}$

Let us write the estimate \eqref{eqn:Avirial2} as
\begin{equation}
\label{eqn:winfbd}
\| w \|_{L^\infty (M_R)} \leq \eta (R) \| w \|_{H^1 (M)},
\end{equation}
for $w \in H^1_r (M)$, under hypothesis \eqref{eqn:Aasymp}, where now
\begin{equation}
\label{eqn:etaAbd}
\eta(R) = \sup_{|r| \geq R} C A (r)^{-1/2}.
\end{equation}

Putting together what we have observed, we see that if \eqref{eqn:intvolel}
or \eqref{eqn:Aasymp} holds, then
\begin{align}
\label{eqn:usobbd}
u \in H^1_\pi (M) & \Rightarrow u |_{M_1} \in L^2 (M_1) \cap L^\infty (M_1) \\
& \Rightarrow u |_{M_1} \in L^q (M_1), \ \forall q \in [2, \infty], \notag
\end{align}
the latter implication by interpolation.  Also, if $2 < q  < \infty$, 
\begin{align}
\label{eqn:uqbd}
\int_{M_R} |u|^q\, d \Vol & \leq \| u \|_{L^\infty (M_R)}^{q-2} \int_{M_R}
|u|^2\, d \Vol \\
& \leq \eta (R)^{q-2} \| u \|_{H^1}^q \notag ,
\end{align}
where $\eta (R)$ is given by \eqref{eqn:etaAbd1} if \eqref{eqn:intvolel}
holds and \eqref{eqn:etaAbd} if \eqref{eqn:Aasymp} holds.  Also, recalling
\eqref{eqn:compemb}, we have
\begin{align}
\label{eqn:usobbd1}
u \in H^1_\pi (M) & \Rightarrow u  \in L^q (M), \quad \forall q \in \left[ 2,
\frac{2n}{n-2} \right) \\
& \Rightarrow u \in L^{p+1} (M), \quad \forall p \in \left( 1, \frac{n+2}{n-2}
\right). \notag
\end{align}
Consequently, $\II (\beta, \pi)$, defined in \eqref{eqn:calI}, is positive.

Let us now take a sequence $(u_\nu)$, 
\begin{equation}
\label{eqn:minseq1}
u_\nu \in H^1_\pi (M) , \quad \| u_\nu \|_{L^{p+1}}^{p+1}
= \beta, \quad F_\lambda
\leq \II (\beta, \pi) + \frac{1}{\nu}.
\end{equation}
Passing to a subsequence, which we continue to denote $(u_\nu)$, we have
\begin{equation}
\label{eqn:wkconv1}
u_\nu \to u \in H^1_\pi (M), \ \text{converging weakly}.
\end{equation}
By \eqref{eqn:Avirial1}, 
\begin{equation}
\label{eqn:stconv}
u_\nu \to u \ \text{in} \ L^{p+1} ( M \setminus M_R)\text{-norm}, \quad
\forall\, R < \infty,
\end{equation}
as long as $p$ satisfies \eqref{eqn:pconstraints}.  If \eqref{eqn:intvolel}
or \eqref{eqn:Aasymp} holds, we have from this and \eqref{eqn:uqbd} that
\begin{equation}
\label{eqn:stpconv}
u_\nu \to u \ \text{in} \ L^{p+1} (M)\text{-norm},
\end{equation}
hence
\begin{equation}
\label{eqn:Jpeqbeta}
J_p (u)  = \beta.
\end{equation}
It follows that 
\begin{equation}
\label{eqn:FlamgeqI}
F_\lambda (u ) \geq \II ( \beta, \pi).
\end{equation}
On the other hand, \eqref{eqn:wkconv1} and \eqref{eqn:FlamH1} imply
\begin{equation}
\label{eqn:FlamleqI}
F_\lambda (u) \leq \II (\beta, \pi).
\end{equation}
 Hence, $u$ is the desired minimizer and we have the following result.
 
 \begin{proposition}
 \label{prop:Flammin}
 To the hypotheses made at the beginning of this section, add that either
 \eqref{eqn:intvolel} or \eqref{eqn:Aasymp} hold. Then the sequence $(u_\nu)$
 in \eqref{eqn:minseq1} has a subsequence that converges in $H^1$-norm to a
 limit $u \in H^1_\pi (M)$ that achieves the minimum in \eqref{eqn:calI}.
 \end{proposition}

\subsection{Energy Minimizers}
\label{sec:vortspherenmin}

As in Section \ref{sec:vortspherFlambda}, we take $M$ to be a complete,
$n$-dimensional, Riemannian manifold, possibly with boundary, with rotational
symmetry.  We desire to minimize the energy
\begin{equation}
\label{eqn:energy2}
E(u) = \frac12 \| \nabla u \|_{L^2}^2 - \frac{1}{p+1} \int_M | u |^{p+1}\, d
\Vol
\end{equation}
over $H^1_\pi (M)$, subject to the constraint
\begin{equation}
\label{eqn:mass2}
Q(u) = \| u \|_{L^2}^2 = \beta,
\end{equation}
with $\beta \in (0,\infty)$.  That is, we seek $u$ achieving
\begin{equation}
\label{eqn:calE}
\EE (\beta, \pi) = \inf \{ E(u) \ : \ u \in H^1_\pi (M), \ Q (u) = \beta \}.
\end{equation}
As for the constraint on $p$, instead of \eqref{eqn:pconstraints}, we assume
the following depending upon the nature of $M$, given as in
\eqref{eqn:M1}--\eqref{eqn:M3}:
\beq
\aligned
\label{eqn:pconstraints_enmin1}
1 < p < 1 + \frac{4}{n}, \quad &\text{if} \ I = [0,\infty), \\
1 < p < \infty, \quad &\text{if} \
I = [1,\infty) \ \text{or} \ (-\infty, \infty).
\endaligned
\eeq
We also assume
\begin{equation}
\label{eqn:Aconstraints}
A(r) \ \text{in \eqref{eqn:volel1} satisfies either \eqref{eqn:intvolel} or
\eqref{eqn:Aasymp}}.
\end{equation}

Since the constraints on $p$ here are at least as strict as in Section
\ref{sec:vortspherFlambda}, Lemma \ref{lem:intvolel} applies, and in concert
with \eqref{eqn:kinetic} and the arguments involving
\eqref{eqn:usobbd}--\eqref{eqn:usobbd1}, we have
\begin{equation}
\label{eqn:usobbd2}
H^1_\pi (M) \hookrightarrow L^{p+1} (M) \ \text{is compact},
\end{equation}
as long as \eqref{eqn:pconstraints_enmin1} holds.  Also, in the case $I =
[0,\infty)$, we have the Gagliardo-Nirenberg inequality
\begin{equation}
\label{eqn:gagnir}
\| u \|_{L^{p+1}} \leq C_\pi \| u \|_{L^2}^{1-\gamma}
\| u \|_{H^1_\pi}^\gamma, \quad \forall \, u \in H^1_\pi (M),
\end{equation}
where
\begin{equation}
\label{eqn:gagnirgamma}
\gamma = \frac{n}{2} - \frac{n}{p+1}.
\end{equation}
Note that \eqref{eqn:gagnir} is a consequence of 
\begin{equation}
\label{eqn:pisobem}
H^1_\pi (M) \subset L^{2 +{4}/({n-2})} (M),
\end{equation}
if $n \geq 3$.  Given the constraint on $p$ in the first part of
\eqref{eqn:pconstraints_enmin1}, where $I = [0, \infty)$,
\begin{equation}
\label{eqn:enmingamma1}
\gamma (p+1) < 2.
\end{equation}
In case $I = [1,\infty)$ or $(-\infty, \infty)$, we have $H^1_\pi (M) \subset
L^\infty (M)$, and hence \eqref{eqn:gagnir} holds for all $\gamma \in (0,1)$,
so for every $p \in (1,\infty)$, we can pick $\gamma \in (0,1)$ such that
\eqref{eqn:gagnir} and \eqref{eqn:enmingamma1} hold.  As a result, we have
\begin{align}
\label{eqn:H1enmaequiv}
\| u \|_{H^1}^2 & = E(u) + \frac{1}{p+1} \| u \|_{L^{p+1}}^{p+1} + Q(u)  \\
& \leq E(u) + \tilde{C} Q(u)^{{(p+1)(1-\gamma)}/{2}}
\| u \|_{H^1}^{\gamma (p+1)} + Q(u) \notag
\end{align}
for all $u \in H^1_\pi (M)$, which gives a priori bounds on $\| u \|_{H^1}$
in terms of bounds on $E(u)$ and $Q(u)$, thanks to \eqref{eqn:enmingamma1}.
For $\EE( \beta, \pi)$ as in \eqref{eqn:calE}, the a priori bounds in
\eqref{eqn:H1enmaequiv} show that for each $\beta \in (0,1)$,
\begin{equation}
\label{eqn:EElowbd}
\EE (\beta, \pi) > -\infty.
\end{equation}

Having \eqref{eqn:EElowbd}, we choose a minimizing sequence $(u_\nu)$ such
that
\begin{equation}
\label{eqn:minseqenmin}
u_\nu \in H^1_\pi (M), \ \| u_\nu \|_{L^2}^2 \equiv \beta, \quad
E( u_\nu) \leq \EE ( \beta, \pi) + \frac{1}{ \nu}.
\end{equation}
In addition, we have the uniform upper bound
\begin{equation}
\label{eqn:H1upbd}
\| u_\nu \|_{H^1} \leq K < \infty.
\end{equation}
Hence, passing to a subsequence, which we continue to denote $(u_\nu)$, we
have
\begin{equation}
\label{eqn:weakstarconv_enmin}
u_\nu \to u \ \text{weak}^* \text{ in} \ H^1_\pi (M).
\end{equation}
Hence, by \eqref{eqn:usobbd2} we have
\begin{equation}
\label{eqn:strconv_enmin}
u_\nu \to u \ \text{in} \ L^{p+1} (M).
\end{equation}
We can actually say a bit more.  Namely, \eqref{eqn:usobbd2} extends to 
\begin{equation}
\label{eqn:usobbd3}
H^1_\pi (M) \hookrightarrow L^q (M) \ \text{is compact}, \quad \forall\,
q \in (2,p+1].
\end{equation}
Hence, \eqref{eqn:weakstarconv_enmin} gives
\begin{equation}
\label{eqn:qstrconv_enmin}
u_\nu \to u \ \text{in} \ L^q (M), \quad \forall\, q \in (2, p+1].
\end{equation}
We also want to complement this with the following result:
\begin{equation}
\label{eqn:2strconv_enmin}
u_\nu \to u \ \text{in norm, in} \ L^2 (M).
\end{equation}
Desirable consequences of \eqref{eqn:2strconv_enmin} arise as follows.  
First, \eqref{eqn:weakstarconv_enmin} implies
\begin{equation}
\label{eqn:nablabd_enmin}
\| \nabla u \|_{L^2} \leq \liminf_{\nu \to \infty} \| \nabla u_\nu \|_{L^2}, 
\end{equation}
hence, by \eqref{eqn:strconv_enmin}, \eqref{eqn:energy2} and
\eqref{eqn:minseqenmin},
\begin{equation}
\label{eqn:Eupbd_enmin}
E(u) \leq \liminf_{\nu \to \infty} E(u_\nu) = \EE ( \beta, \pi).
\end{equation}
Hence, given \eqref{eqn:2strconv_enmin}, we have $\| u \|_{L^2}^2 = \beta$,
so by definition \eqref{eqn:calE}, $E(u) > \EE (\beta, \pi)$.  Comparison
with \eqref{eqn:Eupbd_enmin} gives
\begin{equation}
\label{eqn:Eequal}
E (u) = \EE (\beta, \pi), \ u \in H^1_\pi (M), \ Q(u) = \beta,
\end{equation}
and hence we have the desired energy minimizer.  The result \eqref{eqn:Eequal}
also gives
\begin{equation}
\label{eqn:H1lim}
\| u \|_{H^1} = \lim_{\nu \to \infty} \| u_\nu \|_{H^1},
\end{equation}
which in conjunction with \eqref{eqn:weakstarconv_enmin} gives
\begin{equation}
\label{eqn:H1stconv}
u_\nu \to u \ \text{in} \ H^1_\pi (M).
\end{equation}

Thus, we are left with establishing \eqref{eqn:2strconv_enmin}.  To do this,
we use the concentration-compactness method of P.-L.~Lions \cite{L1}.  As we
are working on spaces that are not necessarily Euclidean, we use the
formulation from \cite{CMMT}, Appendix A.1, which we recall
here for convenience.

\begin{lemma}
\label{lem:conccompman}
Let $X$ be a metric space and $\{ \mu_\nu \}$ a sequence of positive measures
on $X$, each with fixed mass $\beta > 0$.  Then, possibly passing to a
subsequence, one of the following three cases holds:
\begin{enumerate}
\item Vanishing:  If $B_R (y) = \{ x \in X \ : \ \dist (x,y) \leq R \}$, then
for all $R \in (0,\infty)$,
\begin{equation}
\label{eqn:ccvan}
\lim_{\nu \to \infty} \sup_{y \in M} \mu_\nu (B_R (y)) = 0.
\end{equation}

\item Concentration:  There is a sequence $\{ y_\nu \} \subset X$ with the
property that for each $\epsilon > 0$, there exists $R(\epsilon) < \infty$
such that
\begin{equation}
\label{eqn:ccconc}
\mu_\nu ( B_{R(\epsilon)} (y_\nu)) > \beta - \epsilon.
\end{equation}

\item Splitting:  There exists $\alpha \in (0,\beta)$ with the following
properties.  For each $\epsilon > 0$, there exists $\nu_0 \geq 1$ and sets
$E_\nu^\sharp$, $E_\nu^b \subset X$ such that
\begin{equation}
\label{eqn:ccsplit1}
\dist ( E_\nu^\sharp, E_\nu^b) \to \infty \ \text{as} \ \nu \to \infty,
\end{equation}
and 
\begin{equation}
\label{eqn:ccsplit2}
\left|  \mu_\nu (E_\nu^\sharp) - \alpha \right| < \epsilon, \quad
\left| \mu_\nu (E_\nu^b) - (\beta - \alpha) \right| < \epsilon, \quad
\forall\, \nu \geq \nu_0.
\end{equation}

\end{enumerate}
\end{lemma}

\begin{remark}
In \eqref{eqn:ccconc}, $R(\epsilon)$ is independent of $\nu$ and $\{ y_\nu \}$
is independent of $\epsilon$.
\end{remark}

We apply this lemma to the setting where $X = M$ is a complete, Riemannian
manifold with rotational symmetry as defined above, and the measures
$(\mu_\nu)$ are given by
\begin{equation}
\label{eqn:ccmeas_enmin}
\mu_\nu (E) = \int_E | u_\nu |^2 d \Vol.
\end{equation}

We will impose one additional condition on $M$, namely that
\begin{equation}
\label{eqn:Mbdgeom}
M \ \text{has bounded geometry}.
\end{equation}
In such a case, we have the following

\begin{lemma}
\label{lem:van}
Assuming $\{ u_\nu \}$ is a bounded sequence in $H^1 (M)$ and
\begin{equation}
\label{eqn:limsupto0}
\lim_{\nu \to \infty} \sup_{y \in M} \int_{B_R (y)} | u_\nu|^2\, d \Vol = 0,
\ \text{for some} \ R > 0.
\end{equation}
Then,
\begin{equation}
\label{eqn:Lrto0}
2 < r < \frac{2n}{n-2} \Longrightarrow \| u_\nu \|_{L^r (M)} \to 0.
\end{equation}
\end{lemma}

As noted in \cite{CMMT}, this is a special case of
Lemma I.1 on page 231 of \cite{L2}.
In \cite{L2} this lemma was established for $M = \RR^n$, but the only
two geometrical properties used in the proof there are the existence of
Sobolev embeddings on balls of radius $R>0$ and the fact that there exists
$m>0$ such that $\RR^n$ has a covering by balls of radius $R$ in such a way
that each point is contained in at most $m$ balls.  These two properties hold
on every Riemannian manifold with $C^\infty$ bounded geometry.
See \cite{CMMT} for details.

Our next goal is to rule out the property of Alternative $(1)$ (Vanishing)
from Lemma \ref{lem:conccompman}.  To achieve this, we require one further
hypothesis:
\begin{equation}
\label{eqn:EElt0}
\EE (\beta, \pi) < 0,
\end{equation}
or equivalently, we assume there exists $\varphi \in H^1_\pi (M)$ such that
$Q (\phi) = \beta$ and $E(\phi) < 0$.  Note that if we pick $\phi_1 \in
H^1_\pi (M)$ such that $Q (\phi_1) = 1$ and set $\phi_\beta = \beta^{1/2}
\phi_1$, then $Q(\phi_\beta) = \beta$ and
\begin{equation}
\label{eqn:Ebetascaled}
E( \phi_\beta) = \frac{\beta}{2} \| \nabla \phi_1 \|_{L^2}^2 -
\frac{\beta^{({p+1})/{2}}}{p+1} \int_M | \phi_1|^{p+1}\, d \Vol,
\end{equation}
which tends to $-\infty$ as $\beta \to \infty$.  Hence,
\begin{equation}
\label{eqn:calElim}
\lim_{\beta \to \infty} \EE (\beta, \pi) = - \infty,
\end{equation}
so \eqref{eqn:EElt0} holds when $\beta$ is sufficiently large.  Now, we are
prepared to rule out vanishing.

\begin{lemma}
\label{lem:novanishing}
Under the assumptions on $M$ and $\EE (\beta, \pi)$ above, if $\{ u_\nu \}
\subset H^1_\pi (M)$ is an energy minimizing sequence, as in
\eqref{eqn:minseqenmin}, then vanishing as in \eqref{eqn:ccvan} cannot occur.
\end{lemma}

\begin{proof}
Assume \eqref{eqn:ccvan} does not occur.  Then, Lemma \ref{lem:van} gives
\begin{equation}
\label{eqn:Lrto0ap}
\| u_\nu \|_{L^r} \to 0, \quad \forall  \, r \in \left( 2, \frac{2n}{n-2} \right).
\end{equation}
Then, \eqref{eqn:qstrconv_enmin} implies $u = 0$, which implies
\begin{equation}
\label{eqn:Lpto0ap}
\| u_\nu \|_{L^{p+1}} \to 0
\end{equation}
by \eqref{eqn:strconv_enmin}. However, \eqref{eqn:Lpto0ap} and
\eqref{eqn:minseqenmin} imply
\begin{equation}
\label{eqn:nablatoEap}
\frac12 \| \nabla u \|_{L^2}^2 \to \EE (\beta, \pi),
\end{equation}
which is impossible given assumption \eqref{eqn:EElt0}.
\end{proof}

Next, we wish to rule out alternative $(3)$ (splitting) in Lemma
\ref{lem:conccompman}.  Following the methods in \cite{L1,L2}, we proceed via
the following subadditivity result.

\begin{lemma}
\label{lem:subadd}
Under the assumptions of Lemma \ref{lem:novanishing}, if $0 < \eta < \beta$,
then
\begin{equation}
\label{eqn:subadd}
\EE (\beta, \pi) < \EE (\beta - \eta, \pi) + \EE ( \eta, \pi).
\end{equation}
\end{lemma}

\begin{proof}

The proof here is identical to that in \cite{CMMT}, Proposition $3.1.3$.  

\end{proof}

We mention parenthetically that the proof of Lemma \ref{lem:subadd} does
not need either $\EE(\beta-\eta,\pi)<0$ or $\EE(\eta,\pi)<0$.

\begin{lemma}
\label{lem:nosplit}
In the setting of Lemma \ref{lem:subadd}, if $\{ u_\nu \} \subset H^1_\pi (M)$
is an energy minimizing sequence as in \eqref{eqn:minseqenmin}, then splitting
as in \eqref{eqn:ccsplit1}-\eqref{eqn:ccsplit2} with $\mu_\nu$ as in
\eqref{eqn:ccmeas_enmin} cannot occur.
\end{lemma}

\begin{proof}

Assume splitting does occur, in other words
\eqref{eqn:ccsplit1}-\eqref{eqn:ccsplit2} hold, with $\mu_\nu$ as in
\eqref{eqn:ccmeas_enmin}.  We can assume that $E_\nu^\sharp$ and $E_\nu^b$
are invariant under the action of $G$.  Choose $\epsilon > 0$ sufficiently
small such that
\begin{equation}
\label{eqn:EEubeps}
\EE (\beta, \pi) < \EE (\alpha, \pi) + \EE ( \beta - \alpha, \pi) - C_1
\epsilon,
\end{equation}
where $C_1 > 0$ is a sufficiently large constant to be fixed later.
Since $\| u_\nu \|_{H^1 (M)}$ and $\| u_\nu \|_{L^{p+1} (M)}$ are uniformly
bounded, it follows from \eqref{eqn:ccsplit1} that there exists $\nu_1$ such
that $\nu \geq \nu_1$ implies
\begin{equation}
\label{eqn:bound}
\int_{S_\nu} |u_\nu|^2\, d \Vol + \int_{S_\nu} | \nabla u |^2\, d \Vol +
\int_{S_\nu} | u_\nu |^{p+1}\, d \Vol < \epsilon,
\end{equation}
where $S_\nu$ is a set of the form
\begin{equation}
\label{eqn:Snu}
S_\nu = \{ x \in M \ : \ d_\nu < \dist (x, E_\nu^\sharp) \leq d_\nu + 2 \}
\subset M \setminus ( E_\nu^\sharp \cup E_\nu^b),
\end{equation}
for some $d_\nu > 0$.  In other words, 
\begin{equation}
S_\nu = \tilde{E}_\nu (d_\nu + 2) \setminus \tilde{E}_\nu (d_\nu),
\end{equation}
where
\begin{equation}
\label{eqn:tildeEnu}
\tilde{E}_\nu (r) = \{ x \in M \ : \ \dist (x, E_\nu^\sharp) \leq r \}.
\end{equation}
Now, define functions $\chi_\nu^\sharp$ and $\chi_\nu^b$ by
\begin{align}
\chi_\nu^\sharp & = 1, \hspace{3cm} \text{if} \ x \in \tilde{E}_\nu (d_\nu),
\notag \\
 & = 1- \dist (x, \tilde{E}_\nu (d_\nu) ), \quad \text{if} \ x \in
 \tilde{E}_\nu (d_\nu + 1), \label{eqn:chinusharp} \\
 & = 0, \hspace{3cm} \text{if} \ x \notin \tilde{E}_\nu (d_\nu + 2), \notag 
 \end{align}
 and 
\begin{align}
\chi_\nu^b & = 0, \hspace{3cm} \text{if} \ x \in \tilde{E}_\nu (d_\nu + 1),
\notag \\
 & = \dist (x, \tilde{E}_\nu (d_\nu+1) ),
 \quad \text{if} \ x \in \tilde{E}_\nu
 (d_\nu + 2), \label{eqn:chinub} \\
 & = 1, \hspace{3cm} \text{if} \ x \notin \tilde{E}_\nu (d_\nu + 2). \notag 
 \end{align}
 These functions are both Lipschitz with Lipschitz constant $1$ and almost
 disjoint supports.  Also, they are invariant under the action of $G$.  Set
 \begin{equation}
 \label{eqn:unusharp_b}
 u_\nu^\sharp = \chi_\nu^\sharp u_\nu, \ u_\nu^b = \chi_\nu^b u_\nu \in
 H^1_\pi (M).
 \end{equation}
 Note that
 \begin{equation}
 \label{eqn:Qnu}
 Q( u_\nu^\sharp) = \alpha_\nu, \ \ Q (u_\nu^b) = \beta_\nu - \alpha_\nu,
 \end{equation}
 with
 \begin{equation}
 \label{eqn:alphabetabounds}
 | \alpha - \alpha_\nu| < 2 \epsilon, \ \  | (\beta - \alpha) -
 (\beta_\nu - \alpha_\nu) | < 2 \epsilon.
 \end{equation}
 Also,
 \begin{equation}
 \label{eqn:Esepnub}
 E( u_\nu^\sharp + u_\nu^b) = E (u_\nu^\sharp) + E (u_\nu^b).
 \end{equation}
 Meanwhile,
 \begin{equation}
 \label{eqn:usplitnub-p}
 \int_M \Bigl(  |u_\nu|^{p+1} - (|u_\nu^\sharp|^{p+1} + |u_\nu^b|^{p+1})
 \Bigr)\, d \Vol \leq \int_{S_\nu} | u_\nu|^{p+1}\, d \Vol < \epsilon,
 \end{equation}
 and, since $\nabla u_\nu^\sharp = \chi_\nu^\sharp \nabla  u_\nu + ( \nabla
 \chi_\nu^\sharp) u_\nu$, with a similar identity for $\nabla u_\nu^b$, we
 have
  \begin{equation}
  \aligned
 \label{eqn:usplitnub-h1}
 &\int_M \left(  |\nabla u_\nu|^{2} - (| \nabla u_\nu^\sharp|^{2} + |\nabla
 u_\nu^b|^2 \right) d \Vol  \\
 &\leq \int_{S_\nu} | \nabla u_\nu|^{2} d \Vol +
 \int_{S_\nu} | u_\nu|^{2} d \Vol < 2 \epsilon.
 \endaligned
 \end{equation}
 Hence, 
 \begin{equation}
 \label{eqn:Esplitbds}
 | E( u_\nu) - E( u_\nu^\sharp + u_\nu^b) | < 3 \epsilon.
 \end{equation}
 These estimates given, in the limit $\nu \to \infty$,
 \begin{equation}
 \label{eqn:EElbeps}
 \EE (\beta, \pi) \geq \EE ( \alpha, \pi) + \EE (\beta - \alpha, \pi) -
 C_1 \epsilon.
 \end{equation}
 This leads to a contradiction of \eqref{eqn:EEubeps} and proves Lemma
 \ref{lem:nosplit}.
 
\end{proof}

Then, we have the following concentration result
\begin{corollary}
\label{cor:conc}
In the setting of Lemma \ref{lem:novanishing}, there is a sequence
$\{ y_\nu \} \subset M$ with the property that for each $\epsilon > 0$,
there exists $R(\epsilon) < \infty$ such that
\begin{equation}
\label{eqn:concbd}
\int_{M \setminus B_{R(\epsilon)}(y_\nu)} | u_\nu|^2\, d \Vol < \epsilon.
\end{equation}
\end{corollary}

Again, we emphasize that $\{ y_\nu \}$ is independent of $\epsilon$, and
$R(\epsilon)$ is independent of $\nu$.

To proceed from concentration to
compactness, we use the symmetry condition $u_\nu \in H^1_\pi (M)
\Rightarrow | u_\nu | \in H^1_r (M)$, and we bring in the following hypothesis:
\begin{equation}
\label{eqn:diamgrowth}
\diam S_\rho \to \infty \ \text{as} \ | \rho | \to \infty,
\end{equation}
where $S_\rho$ denotes the set of points $x \in M$ of the form $x = (\rho,
\omega)$, given $M$ of the form \eqref{eqn:M1}--\eqref{eqn:M3}.  (We will say
$r(x) = \rho$).  Note that \eqref{eqn:diamgrowth} follows from
\eqref{eqn:Aasymp}, but not necessarily from \eqref{eqn:intvolel}
(though it might follow from \eqref{eqn:intvolel}
together with \eqref{eqn:Mbdgeom}).

\begin{lemma}
\label{lem:rbounded}
In the setting of Corollary \ref{cor:conc}, add hypothesis
\eqref{eqn:diamgrowth}.  Then there exists $R_0 < \infty$ such that
\begin{equation}
\label{eqn:rbounded}
| r( y_\nu)| \leq R_0, \ \forall \nu.
\end{equation}
\end{lemma}

\begin{proof}
Take $\epsilon_0 < \beta/2$ and then take $R_0$ so large that
\begin{equation}
\label{eqn:diamlb}
\rho \geq R_0 \Rightarrow \diam S_\rho > 2 R (\epsilon_0).
\end{equation}
If $| r( y_\nu)| > R_0$, take $g \in G$ such that $y_\nu' = g \cdot y_\nu$
satisfies $\dist (y_\nu ' ) > 2 R( \epsilon_0)$, so
\begin{equation}
\label{eqn:disjoint}
B_{R (\epsilon_0)} ( y_\nu ' ) \cap B_{R (\epsilon_0)} ( y_\nu) = \emptyset .
\end{equation}
Now, the identity
\begin{equation}
\label{eqn:L2equal}
\int_{B_{R (\epsilon_0)} (y_\nu)} | u_\nu|^2\, d \Vol = \int_{B_{R (\epsilon_0)}
(y_\nu ')} | u_\nu|^2\, d \Vol
\end{equation}
contradicts \eqref{eqn:concbd}.   Thus, \eqref{eqn:diamlb} must hold.
\end{proof}

\begin{corollary}
\label{cor:conc1}
In the setting of Lemma \ref{lem:rbounded}, for each $\epsilon > 0$ there
exists a compact $K (\epsilon) \subset M$ such that
\begin{equation}
\int_{M \setminus K(\epsilon)} | u_\nu |^2\, d \Vol < \epsilon, \quad
\forall\, \nu.
\end{equation}
\end{corollary}

Now, we can state the main result of this section.

\begin{proposition}
\label{prop:enminconc}
Under the hypotheses of Lemma \ref{lem:rbounded}, the sequence $\{ u_\nu \}
\subset H^1_\pi (M)$ in \eqref{eqn:minseqenmin} has a subsequence that
converges in $H^1_\pi$ norm to $u \in H^1_\pi (M)$ satisfying $Q(u) = \beta$
and $E(u) = \EE ( \beta, \pi)$.
\end{proposition}

\begin{proof}
It suffices at this point to show that $(u_\nu)$ has a subsequence (continued
to be called $(u_\nu)$) that converges in $L^2$-norm.  In view of
\eqref{eqn:weakstarconv_enmin}-\eqref{eqn:qstrconv_enmin}, the desired
$L^2$ convergence is a simple consequence of Corollary \ref{cor:conc1}.
\end{proof}

Such a minimizer satisfies the PDE \eqref{eqn:nlbs}.  In fact, complementing
\eqref{eqn:nlbs}, we have for $u \in H^1_\pi (M)$,
\begin{align}
\label{eqn:dE}
\frac{d}{d \tau} E (u + \tau \psi) \Bigr|_{\tau = 0}  & =
\Re ( - \Delta u - |u|^{p-1} u, \psi), \\
\label{eqn:dQ}
\frac{d}{d \tau} Q (u + \tau \psi) \Bigr|_{\tau = 0} & = 2 \Re  (u, \psi) .
\end{align}
Hence, given a constrained minimizer $u \in H^1_\pi (M)$ produced by
Proposition \ref{prop:enminconc}, then
\begin{equation}
\label{eqn:EL1}
\psi \in H^1_\pi (M), \ \Re (u, \psi) = 0 \Rightarrow \Re ( \Delta u +
| u |^{p-1} u , \psi) = 0,
\end{equation}
and it follows that there exists $\lambda \in \RR$ such that $ \Delta u +
| u |^{p-1} u =  \lambda u$, or equivalently
\begin{equation}
\label{eqn:nlellip}
-\Delta u + \lambda u = | u |^{p-1} u.
\end{equation}

\subsection{Weinstein Functional Maximizers}
\label{sec:vortspherWmax}

Here we take $M = \RR^n$ and $G \subset SO (n)$ a group acting transitively
on the unit sphere $\SS^{n-1}$.  Given a unitary representation $\pi$ of $G$
on $V$, we define $H^1_\pi (\RR^n)$ as in \eqref{eqn:Hspi}.  We seek to
maximize
\begin{equation}
\label{eqn:WW}
W(u) = \frac{ \| u \|_{L^{p+1}}^{p+1} }{ \| u \|_{L^2}^\alpha \| \nabla u
\|_{L^2}^\beta},
\end{equation}
over non-zero $u \in H^1_\pi (\RR^n)$, given
\begin{equation}
\label{eqn:alphabetaW}
\alpha = 2- \frac{(n-2)(p-1)}{2}, \quad \beta = \frac{n (p-1)}{2},
\end{equation}
assuming $p$ satisfies
\begin{equation}
\label{eqn:pboundsW}
1 < p < \frac{n+2}{n-2}.
\end{equation}
Note that in such a case,
\begin{equation}
\label{eqn:alphabetapW}
\alpha, \beta > 0, \ \ \alpha + \beta = p+1.
\end{equation}
The supremum
\begin{equation}
\label{eqn:Wpi}
\mathcal{W}  ( \pi) = \sup \{ W(u) \ : \ 0 \neq u \in H^1_\pi ( \RR^n) \}
\end{equation}
is the best constant in the Gagliardo-Nirenberg inequality
\begin{equation}
\label{eqn:gagnirpiW}
\| u \|_{L^{p+1}}^{p+1} \leq \mathcal{W} (\pi) \| u \|_{L^2}^\alpha \| \nabla
u \|_{L^2}^\beta, \ \ u \in H^1_\pi (\RR^n).
\end{equation}
That there exists $\tilde{\mathcal{W}} < \infty$
such that $\| u \|_{L^{p+1}}^{p+1}
\leq \tilde{\mathcal{W}} \| u \|_{L^2}^\alpha \| \nabla u \|_{L^2}^\beta$
follows from the classical Gagliardo-Nirenberg inequality.  It follows that
\eqref{eqn:Wpi} exists and is finite, for each finite-dimensional unitary
representation $\pi$ of $G$.  We aim to show that such a supremum is obtained.

To proceed, take $u_\nu \in H^1_\pi(\RR^n)$
such that $W( u_\nu) \to \mathcal{W} (\pi)$.
We follow the standard argument in the radial case and use the fact that $W(u)$
is invariant under $u \to a u$ and $u(x) \to u(bx)$ to impose the normalization
\begin{equation}
\label{eqn:L2H1eq1W}
\| u_\nu \|_{L^2} = 1, \ \ \| \nabla u_\nu \|_{L^2} = 1,
\end{equation}
so
\begin{equation}
\label{eqn:LptoW}
\| u_\nu \|_{L^{p+1}} \to \mathcal{W} (\pi)^{{1}/({p+1})}.
\end{equation}
If we pass to a subsequence such that $u_\nu \to u,\ \text{weak}^*$ in
$H^1_\pi (\RR^n)$, results from Section
\ref{sec:vortspherFlambda} yield $u_\nu \to u$ in norm in
$L^{p+1} (\RR^n)$.  Also, $\| u \|_{L^2} \leq 1$ and $\| \nabla u \|_{L^2}
\leq 1$, so
\begin{equation}
\label{eqn:Wlbd}
W(u) \geq \mathcal{W} (\pi).
\end{equation}
This yields $W(u) = \mathcal{W} (\pi)$ (hence $\| u \|_{L^2} = \| \nabla u
\|_{L^2} = 1$, and therefore $u_\nu \to u$ in norm in $H^1_\pi (\RR^n)$).
We have the desired maximizer.
We summarize.

\begin{proposition} \label{p2.3.a}
Given \eqref{eqn:pboundsW}, there exists a nonzero $u\in H^1_\pi(\Rr^n)$
maximizing $W(u)$ in \eqref{eqn:Wpi}, so
\beq
W(u)=\Cal{W}(\pi).
\eeq
\end{proposition}

A computation of
\begin{equation}
\label{eqn:dWdtau}
\frac{d}{d \tau} W(u + \tau v) \Bigr|_{\tau = 0} = \frac{  (N(u), v)}{ \| u
\|_{L^2}^{2 \alpha} \| \nabla u \|_{L^2}^{2 \beta} }
\end{equation}
shows that such a maximizer $u$ solves the equation
\begin{equation}
\label{eqn:ELW}
\Delta u - \lambda u + K | u|^{p-1} u = 0,
\end{equation}
with
\begin{equation}
\label{eqn:lamKdefW}
\lambda = \frac{\alpha}{ \beta}\, \frac{\|\nabla u\|_{L^2}^2}{\|u\|_{L^2}^2},
\quad K = \frac{p+1}{\beta}\,
\frac{\|\nabla u\|^2_{L^2}}{\|u\|^{p+1}_{L^{p+1}}},
\end{equation}
hence, with the normalization \eqref{eqn:L2H1eq1W}, 
\begin{equation}
\label{eqn:lamKdefW1}
\lambda = \frac{\alpha}{ \beta}, \ \ K = \frac{p+1}{\beta \mathcal{W} (\pi)}.
\end{equation}

We next examine the dependence of $\mathcal{W} (\pi)$ on the representation
$\pi$ of $G$ on $V$, and establish a certain monotonicity property.  Let us
assume that
\begin{equation}
\label{eqn:piirred}
\pi \ \text{is irreducible and} \ \dim V_0 =1.
\end{equation}
See Appendices \ref{app:H1pineq0}--\ref{app:4d} for results on when
\eqref{eqn:piirred} holds.  In such a case, pick a unit spanning vector
$\pi \in V_0$.  Then
\begin{equation}
\label{eqn:urgp0W}
u(r g \cdot p_0) = \psi (r) \pi (g) \phi, \ \ \psi: \RR^+ \to \CC.
\end{equation}
We have 
\begin{equation}
\label{eqn:u0W}
u_0 (x) = \phi (r) , \ r=|x| \Longrightarrow u_0 \in H^1_r (\RR^n),
\end{equation}
and
\begin{equation}
\label{eqn:uL2Lp}
\| u \|_{L^2} = \| u_0 \|_{L^2}, \ \ \| u \|_{L^{p+1}} = \| u_0 \|_{L^{p+1}}.
\end{equation}
To compare $\| \nabla u \|_{L^2}$ with $\| \nabla u_0 \|_{L^2}$, note that
\begin{equation}
\label{eqn:kineticW}
\| \nabla u_0 \|_{L^2}^2 = - ( \Delta u, u)
\end{equation}
and 
\begin{equation}
\label{eqn:radLapW}
\Delta u = \pa_r^2 u + \frac{n-1}{r} \pa_r u + \frac{1}{r^2} \Delta_{\SS} u.
\end{equation}
We preview a result that will be discussed in more
detail in Section \ref{sec:vortspherODE},
namely that under the hypotheses \eqref{eqn:piirred}--\eqref{eqn:urgp0W},
\begin{equation}
\label{eqn:sphereig}
\Delta_{\SS} u = -\mu_\pi^2 u, \ \ \mu_\pi^2 \in \RR^+.
\end{equation}
Hence, 
\beq
\aligned
\| \nabla u \|_{L^2}^2 & = - (\Delta u, u)  \\
& = \| \nabla u_0 \|_{L^2}^2 + \mu_\pi^2 \left\| \frac{u_0}{|x|}
\right\|_{L^2}^2.
\label{eqn:kinsphdecomp}
\endaligned
\eeq
It follows that 
\beq
\aligned
\label{eqn:Winv}
\frac{1}{\mathcal{W} (\pi) } & = \inf \left\{ \frac{ \| u \|_{L^2}^\alpha \|
\nabla u \|_{L^2}^\beta}{ \| u \|_{L^{p+1}}^{p+1} } \ : \ 0 \neq u \in
H^1_\pi (\RR^n) \right\}  \\
& = \inf \left\{ \frac{  \| u_0 \|_{L^2}^\alpha ( \| \nabla u_0 \|_{L^2}^2 +
\mu_\pi^2 \| u_0 / |x| \|_{L^2}^2 )^{\beta/2} }{ \| u_0 \|_{L^{p+1}}^{p+1}}
\ : \ u_0 \in H^1_r (\RR^n) \right\} .
\endaligned
\eeq
The right side of \eqref{eqn:Winv} is clearly a monotone increasing function
of $\mu_\pi^2$.  We have the following conclusion.

\begin{proposition}
\label{prop:Wmax}
Given $p$ as in \eqref{eqn:pboundsW}, $\mathcal{W} (\pi)$ as in
\eqref{eqn:Wpi}, if $\pi_1$ and $\pi_2$ are two representations of $G$
satisfying \eqref{eqn:piirred}, then
\begin{equation}
\label{eqn:monpi}
\mu_{\pi_2}^2 > \mu_{\pi_1}^2 \Longrightarrow
\mathcal{W} (\pi_2) < \mathcal{W} (\pi_1).
\end{equation}
\end{proposition}

This extends Corollary $16$ of \cite{FG}, which deals with $2D$ vortices.
We note parenthetically the $\pi$-dependence in
\beq
\aligned
\mathcal{I}(\beta,\pi)
&=\inf \left\{F_\lambda(u):u\in H^1_\pi(\RR^n),\ J_p(u)=\beta\right\} \\
&=\inf\Bigl\{F_\lambda(u_0)+\mu_\pi^2\Bigl\|\frac{u_0}{|x|}\Bigr\|^2_{L^2}:
u_0\in H^1_r(\RR^n),\ J_p(u_0)=\beta\Bigr\},
\endaligned
\label{W.24}
\eeq
and
\beq
\aligned
\mathcal{E}(\beta,\pi)
&=\inf \left\{ E(u):u\in H^1_\pi(\RR^n),\ \|u\|^2_{L^2}=\beta\right\} \\
&=\inf\Bigl\{E(u_0)+\frac{\mu_\pi^2}{2}\Bigl\|\frac{u_0}{|x|}\Bigr\|^2_{L^2}:
u_0\in H^1_r(\RR^n),\ \|u_0\|^2_{L^2}=\beta\Bigr\}.
\endaligned
\label{W.25}
\eeq
It follows that, for fixed $\beta, \lambda$, and appropriate bounds on $p$,
\beq
\mu_{\pi_2}^2>\mu_{\pi_1}^2\Longrightarrow \mathcal{I}(\beta,\pi_2)>
\mathcal{I}(\beta,\pi_1)\ \text{and}\ \mathcal{E}(\beta,\pi_2)>
\mathcal{E}(\beta,\pi_1).
\label{W.26}
\eeq
For more on this, in a more general context, see Section \ref{s4.2}.

\subsection{ODE for vortex standing waves}
\label{sec:vortspherODE}

We start with
\begin{equation}
\label{eqn:MEuc}
M = \RR^n, \ \text{with the standard flat metric},
\end{equation}
and as in \eqref{eqn:Hspi}, take
\begin{equation}
\label{eqn:H1pi2}
H^1_\pi(M) = \{ u \in H^1 (M, V) \ : \ u(g x) = \pi (g) u(x), \
\forall\, x \in M,\, g \in G \},
\end{equation}
where $G \subset SO (n)$ is a Lie group acting transitively on $\SS^{n-1}$ as
a group of isometries and $\pi$ is a unitary representation of $G$ on a
finite-dimensional inner-product space $V$.  Later in this subsection we
consider more general $M$.  The Laplace operator on $\RR^n$ has the form
\begin{equation}
\label{eqn:radLapODE}
\Delta u = \pa_r^2 u + \frac{n-1}{r} \pa_r u + \frac{1}{r^2} \Delta_{\SS} u, 
\end{equation}
where $\Delta_\SS$ is the Laplace-Beltrami operator on the sphere $\SS^{n-1}$,
with its standard metric.  Thus the PDE
\begin{equation}
\label{eqn:ellipODE}
-\Delta u + \lambda u - |u|^{p-1} u = 0,
\end{equation}
solved by an $F_\lambda$-minimizer, an energy minimizer, or a Weinstein
functional maximizer, takes the form
\begin{equation}
\label{eqn:nlbsODE}
\pa_r^2 u + \frac{n-1}{r} \pa_r u + \frac{1}{r^2} \Delta_{\SS} u - \lambda u
+ |u|^{p-1} u = 0.
\end{equation}
We want to reduce this to an ODE and make some observations.

We proceed as follows.  An element $u \in H^1_\pi (M)$ is uniquely specified
as
\begin{equation}
\label{eqn:urgp0ODE}
u(r, g \cdot p_0) = \pi (g) v(r),
\end{equation}
with
\begin{equation}
\label{eqn:vdefODE}
v: \RR^+ \to V_0 = \{ \phi \in V \ : \ \pi (k) \phi = \phi, \ \forall k \in
K \},
\end{equation}
where $p_0 \in \SS^{n-1}$ is chosen, and $K$ is the subgroup of $G$ fixing
$p_0$.  Let us assume
\begin{equation}
\label{eqn:piirredODE}
\pi \ \text{acts irreducibly on} \ V, \ \text{and} \ V_0 \neq 0.
\end{equation}
We know that (up to unitary equivalence), $V$ is a linear subspace of
$L^2 (\SS^{n-1})$ and $\pi$ acts as the regular representation
\begin{equation}
\label{eqn:regrepODE}
R(g) f(x) = f(g^{-1} x),\quad g \in G.
\end{equation}
In particular, $\pi$ commutes with the action of the Laplace-Beltrami
operator $\Delta_{\SS}$, so
\begin{equation}
\label{eqn:SSspec}
f \in V \Longrightarrow \Delta_{\SS} f = -\mu_\pi^2 f,
\quad \mu_\pi^2 \in \RR^+.
\end{equation}
Thus, \eqref{eqn:radLapODE} yields the ODE
\begin{equation}
\label{eqn:vortexODE}
v'' (r) + \frac{n-1}{r} v' (r) - \frac{ \mu_\pi^2}{r^2} v(r) + | v(r) |^{p-1}
v(r) = 0.
\end{equation}
This is a second order, nonlinear, $k \times k$ system, with $k = \dim V_0$.  

Let us now assume
\begin{equation}
\label{eqn:dimV0ODE}
\dim V_0 = 1,
\end{equation}
which is frequently (but not always) the case.  Then, \eqref{eqn:vortexODE}
is a scalar equation.  That is, we can pick a unit vector $\phi \in V_0$,
spanning $V_0$, and write
\begin{equation}
\label{eqn:vphi}
v(r) = \psi(r) \phi, \quad \psi: \RR^+ \to \CC,
\end{equation}
and \eqref{eqn:vortexODE} becomes
\begin{equation}
\label{eqn:psiODE}
\psi''(r) + \frac{n-1}{r} \psi' - \frac{ \mu_\pi^2}{r^2} \psi - \lambda \psi
+ | \psi |^{p-1} \psi = 0.
\end{equation}
Note that if $\psi (r)$ satisfies \eqref{eqn:psiODE} and $\theta \in \RR$,
then $e^{i \theta} \psi (r)$ also solves \eqref{eqn:psiODE}.  Hence,
$\psi (1)$ can be taken to be real.  If $\psi'(1)$ is also real, then the
solution is real by ODE uniqueness theory.  If $\psi' (1)$ is not real
(and $\psi (0) \in \RR \setminus \{ 0 \}$) then the solution cannot be so
modified to be real.  We show that if $u$ solves \eqref{eqn:ellipODE} by
virtue of being a $F_\lambda$-minimizer (or an energy minimizer or a
Weinstein functional maximizer) then $\psi$ can be arranged to be real.
To see this, use \eqref{eqn:urgp0ODE} and \eqref{eqn:vphi}, i.e.,
\begin{equation}
\label{eqn:upsipiODE}
u( r, g \cdot p_0) = \psi (r) \pi (g) \phi,
\end{equation}
to define the map $u \to u^\sharp$ on $H^1_\pi (M)$:
\begin{equation}
\label{eqn:sharpmap}
u^\sharp  (r, g \cdot p_0) = | \psi (r) | \pi (g) \phi.
\end{equation}
It is apparent that, for a.e. $x \in M$,
\begin{equation}
\label{eqn:sharpkin}
| u^\sharp (x) | = |u(x)|, \quad
| \nabla u^\sharp (x) | \leq | \nabla u (x) |.
\end{equation}
Hence,
\begin{equation}
\label{eqn:sharpnorms}
\| u^\sharp \|_{L^2} = \| u \|_{L^2}, \quad \| u^\sharp \|_{L^{p+1}} = \| u
\|_{L^{p+1}}, \quad \| \nabla u^\sharp \|_{L^2} \leq \| \nabla u \|_{L^2},
\end{equation}
so if $u \in H^1_\pi (M)$ is an $F_\lambda$-minimizer (resp., energy
minimizer, etc.) so is $u^\sharp$.  Hence, $u^\sharp$ solves
\eqref{eqn:ellipODE} and is real valued, in fact $\geq 0$.  More precisely,
\begin{equation}
\label{eqn:psisharppos}
\psi^\sharp > 0 \ \text{on} \ (0,\infty).
\end{equation}
In fact, if $r_0 \in (0,\infty)$ and $\psi^\sharp (r_0 ) = 0$, then
$\psi^\sharp$ achieves a minimum at $r_0$, so $\pa_r \psi^\sharp (r_0)=0$.
Then, uniqueness for solutions to \eqref{eqn:psiODE} would force
$\psi^\sharp \equiv 0$.  Of course, when $u \in H^1_\pi (M)$ is continuous
and $\pi$ satisfies \eqref{eqn:piirredODE} and is not the trivial
representation, $u^\sharp (0) = 0$ and hence
\begin{equation}
\label{eqn:psisharp0}
\psi^\sharp (0) = 0.
\end{equation}

As noted above, after multiplying $u$ by a constant $e^{i \theta}$, we can
arrange that \eqref{eqn:upsipiODE} holds with $\psi(1) > 0$ (now that we know
that $\psi(1) \neq 0$).  With $\psi^\sharp$ and $u^\sharp$ as above, we see
that if $\psi$ is not real for all $r \in \RR^+$, then
\begin{equation}
\label{eqn:radkinlb}
\int_M | \pa_r u |^2\, d\Vol > \int_M | \pa_r u^\sharp |^2\, d \Vol,
\end{equation}
hence
\begin{equation*}
\| \nabla u \|_{L^2}^2 > \| \nabla u^\sharp \|_{L^2}^2.
\end{equation*}
Since this contradicts the minimizing property of $u$, we have the following.
\begin{proposition}
\label{prop:psireal}
Take $M = \RR^n$ and assume \eqref{eqn:piirredODE} and \eqref{eqn:dimV0ODE}
hold.  Let $u \in H^1_\pi (M)$ be an $F_\lambda$-minimizer solving
\eqref{eqn:ellipODE} (or an energy minimizer, etc.) under appropriate
hypotheses on $p$.  Take a unit $\phi \in V_0$, so \eqref{eqn:upsipiODE}
holds.  Then there exists a constant $e^{i \theta}$ such that if $u$ is
replaced by $e^{i \theta} u$ (relabeled $u$), we have $\psi > 0$ on
$(0,\infty)$.
\end{proposition}

Moving on, let us replace the irreducibility hypothesis \eqref{eqn:piirredODE} by
\begin{align}
\label{eqn:altpidefODE}
\pi = \pi_1 \oplus \cdots \oplus \pi_k, & \ \ \ \pi_j \ \text{irreducible on}
\ V_j \subset V, \\
V_{j_0} \neq 0, \ \forall j, \ \ & \ \mu_{\pi_1} = \cdots = \mu_{\pi_k} =:
\mu_\pi. \notag
\end{align}
Then we continue to get the ODE \eqref{eqn:psiODE}, with $u$ and $v$ related
by \eqref{eqn:upsipiODE}.  Next, we replace \eqref{eqn:dimV0ODE} by
\begin{equation}
\label{eqn:linindODE}
v(r_0) \ \text{and} \ v' (r_0) \in V_0 \ \text{are linearly dependent},
\end{equation}
for some $r_0 \in (0,\infty)$.  Note that $V_0 = V_{10} \oplus \cdots \oplus
V_{k0}$; we are not assuming $\dim V_{j0} = 1$.  It follows from
\eqref{eqn:linindODE} that there exists a unit $\phi \in V_0$ such that
$v(r_0)$ and $v' (r_0)$ are scalar multiples of $\phi$.  The scalar nature
of \eqref{eqn:vortexODE} then implies that $v(r)$ has the form
\eqref{eqn:vphi} for all $r \in \RR^+$, with $\psi$ satisfying
\eqref{eqn:psiODE}.  From here, the rest of the argument yielding Proposition
\ref{prop:psireal} applies.  We have the following.
\begin{proposition}
\label{prop:psireal1}
Take $M = \RR^n$ and assume \eqref{eqn:altpidefODE} and \eqref{eqn:linindODE}
hold.  Let $u \in H^1_\pi (M)$ be an $F_\lambda$-minimizer solving
\eqref{eqn:ellipODE} (or an energy minimizer, etc.) under appropriate
hypotheses on $p$.  Then there exists $\phi \in V_0$ such that
\eqref{eqn:upsipiODE} holds, and there exists a constant $e^{i \theta}$ such
that if $u$ is replaced by $e^{i \theta} u$ (relabeled $u$), we have $\psi >
0$ on $(0,\infty)$.
\end{proposition}

\begin{corollary}
\label{cor:psireal1}
Take $M = \RR^n$ and assume \eqref{eqn:altpidefODE} holds.  Let $u \in
H^1_\pi (M)$ be an $F_\lambda$-minimizer solving \eqref{eqn:ellipODE}
(or an energy minimizer, etc.), under appropriate hypotheses on $p$.
Then $u$ cannot vanish anywhere on $\RR^n \setminus \{ 0 \}$.
\end{corollary}

\begin{proof}

If $x \neq 0$ and $u(x) = 0$, then $v(r_0) = 0$ for $r_0 = |x|$, so
\eqref{eqn:linindODE} holds.  Hence, Proposition \ref{prop:psireal1}
applies, yielding a contradiction.

\end{proof}

\begin{question}
With or without hypothesis \eqref{eqn:piirredODE} or \eqref{eqn:altpidefODE},
for a solution $u \in H^1_\pi (\RR^n)$ to \eqref{eqn:ellipODE}, obtained as
an $F_\lambda$-minimizer(or an energy minimizer, etc.) under appropriate
hypotheses on $p$, and with $v$ as in \eqref{eqn:urgp0ODE}, is it possible
for $v(r_0)$ and $v' (r_0)$ to be linearly independent?
\end{question}

We move on to more general Riemannian manifolds with rotational symmetry,
as in \eqref{eqn:M1}--\eqref{eqn:M3}, with metric tensor as in
\eqref{eqn:metric1}, i.e.,
\begin{equation}
\label{eqn:metricarray}
g = \left(  \begin{array}{cc} 
1 &  \\
  &  h(r) 
 \end{array} \right).
 \end{equation}
 We have, in place of \eqref{eqn:radLapODE},
 \begin{equation}
 \label{eqn:metLapODE}
 \Delta u = \gamma^{-1/2} \pa_j ( \gamma^{1/2} g^{jk} \pa_k u ),
 \end{equation}
 with
 \begin{equation}
 \label{eqn:gammadef}
 \gamma = \det g = \det h.
 \end{equation}
 Consequently,
 \begin{equation}
 \label{eqn:metellip}
 \Delta u = \pa_r^2 u + \gamma^{-1/2} ( \pa_r \gamma^{1/2} ) \pa_r u
 + \Delta_{h(r)} u.
 \end{equation}
 Note that if $M = \RR^n$ with its standard metric, then $h(r) = r^2 h(1)$,
 so $\gamma = r^{2(n-1)} \det h(1)$, and hence $\gamma^{-1/2}
 (\pa_r \gamma^{1/2}) = (n-1)/{r}$, as in \eqref{eqn:radLapODE}.
 More generally, since $(\det h(r))^{1/2}$ is the area density of
 $\SS^{n-1}$ with metric tensor $h(r)$, we have
 \begin{equation}
 \label{eqn:gammaArel}
 \gamma^{-1/2} (\pa_r \gamma^{1/2}) = \frac{ A' (r) }{A(r)},
 \end{equation}
 where $A(r)$ is the $(n-1)$-dimensional area of $\SS^{n-1}$ with metric
 tensor $h(r)$, as in \eqref{eqn:volel1}.  Thus, we have
 \begin{equation}
 \label{eqn:LapEucA}
 \Delta u = \pa_r^2 u + \frac{ A' (r) }{A(r)} \pa_r u + \Delta_{h(r)} u.
 \end{equation}
 By \eqref{eqn:intvolel}, 
 \begin{equation}
 \label{eqn:LapHypA}
 \frac{A'(r)}{A(r)} = (n-1) \frac{ \cosh r}{ \sinh r}, \quad \text{if} \ M =
 \HH^n.
 \end{equation}
 
 The sphere $\SS^{n-1}$ has only one metric tensor invariant
 under $SO(n)$, up to a constant factor, so if $G = SO(n)$, we must have
 \begin{equation}
 \label{eqn:hdecomp}
 h(r) = \sigma(r)^2 h_\SS, \ \ \sigma : I \to (0,\infty),
 \end{equation}
 where $I$ is $(0,\infty)$, $[1,\infty)$ or $(-\infty, \infty)$, and $h_\SS$
 is the standard metric on the unit sphere in $\RR^n$.  Note that
 \begin{equation}
 \label{eqn:Asigma}
 A (r) = \sigma(r)^{n-1} A_n,
 \end{equation}
 so $A'(r)/A(r) = (n-1) \sigma' (r)/ \sigma (r)$.  Also,
 \begin{equation}
 \label{eqn:hLap}
 \Delta_{h(r)} = \frac{1}{\sigma(r)^2} \Delta_\SS,
 \end{equation}
 with $\Delta_\SS$ as in \eqref{eqn:radLapODE}.  Consequently, when $G = SO(n)$,
 we have
 \begin{equation}
 \label{eqn:sigmaLap}
 \Delta u = \pa_r^2 u  + (n-1) \frac{\sigma'(r)}{\sigma(r)} \pa_r u +
 \frac{1}{\sigma(r)^2} \Delta_\SS u.
 \end{equation}
 Let $u \in H^1_\pi (M)$ solve \eqref{eqn:ellipODE}, and take $v$ as in
 \eqref{eqn:vdefODE}.  Under hypothesis \eqref{eqn:piirredODE} on the
 representation $\pi$, which implies \eqref{eqn:SSspec}, or more generally
 under hypothesis \eqref{eqn:altpidefODE}, we get from \eqref{eqn:sigmaLap}
 the ODE
 \begin{equation}
 \label{eqn:sigmaODE}
 v''(r) + (n-1) \frac{ \sigma' (r)}{\sigma (r)} v' (r)
 - \frac{ \mu_\pi^2}{\sigma(r)^2} v(r) - \lambda v(r) + |v(r)|^{p-1} v(r) = 0,
 \end{equation}
 in place of \eqref{eqn:vortexODE}, as long as \eqref{eqn:hLap} holds.
 
 We now turn to the case where $G$ is a proper subgroup of $SO(n)$, acting
 transitively on $\SS^{n-1}$, such as $SU (m)$ or $U(m)$, when $n = 2m$.
 In such a case, \eqref{eqn:hdecomp} need not hold, and the various operators
 $\Delta_{h(r)}$  need not be multiples of each other.  We still have the
 following.  Assume the irreducibility condition \eqref{eqn:piirredODE}.
 Then, we can assume $V \subset L^2 ( \SS^{n-1})$ and $\pi$ is given by
 \eqref{eqn:regrepODE}.  In such a case, $\pi$ commutes with $\Delta_h$ for
 each $G$-invariant metric tensor $h$ on $\SS^{n-1}$.  Hence $\pi$ commutes
 with $\Delta_{h(r)}$ for each $r$, so extending \eqref{eqn:SSspec}, we have
 \begin{equation}
 \label{eqn:SSspec1}
 f \in V \Rightarrow \Delta_{h(r)} f = -\mu_\pi (r)^2 f,
 \quad \mu_\pi (r)^2 \in \RR^+.
 \end{equation}
 Hence, \eqref{eqn:LapEucA} gives
 \begin{equation}
 \label{eqn:LapEucAalt}
 \Delta u = \pa_r^2 u + \frac{A'(r)}{A(r)} \pa_r u - \mu_\pi (r)^2 u, 
 \end{equation}
 so if $u \in H^1_\pi (M)$ solves \eqref{eqn:ellipODE} and we take $v$ as in
 \eqref{eqn:vdefODE}, then under hypothesis \eqref{eqn:piirredODE} we get the
 ODE
 \begin{equation}
 \label{eqn:vortexODERie}
 v''(r) + \frac{A'(r)}{A(r)} v'(r) - \mu_\pi (r)^2 v(r) - \lambda v(r) + |
 v(r)|^{p-1} v = 0.
 \end{equation}

\section{Axial Vortices}
\label{sec:vortaxial}

Here we take $G \subset SO(n)$, acting transitively on $\SS^{n-1}$, then
acting on $\RR^{n+k}$ by
\begin{equation}
\label{eqn:axialmetten}
\left( \begin{array}{cc}
g & \\
 & I 
 \end{array} \right), 
 \end{equation}
 where $I$ is the $k \times k$ identity matrix.  Given a unitary
 representation $\pi$ of $G$ on $V$, we define $H^1_\pi ( \RR^{n+k})$ as in
 \eqref{eqn:Hspi}.  We call elements of such a space axial vortices.
 
 In Section \ref{sec:vortaxialFlambda}, we seek to minimize $$F_\lambda (u) =
 \| \nabla u \|_{L^2}^2 + \lambda \| u \|_{L^2}^2$$ over $u \in H^1_\pi
 (\RR^{n+k})$, subject to the constraint $J_p (u) = \| u \|_{L^{p+1}}^{p+1}
 = \beta$, with $\beta \in (0, \infty)$ given.  We assume
 \begin{equation}
 \label{eqn:axialpbds}
 1 < p < \frac{n+k+2}{n+k-2}, \quad \lambda > 0.
 \end{equation}
This guarantees that $F_\lambda \equiv \| u \|_{H^1}^2$ and note that
\begin{equation}
\label{eqn:axialsobem}
H^1 (\RR^{n+k}) \subset L^{p+1} ( \RR^{ n+k} ) , \quad \forall\,
p \in \left( 1,\frac{n+k+2}{n+k-2} \right).
\end{equation}
Calculations parallel to \eqref{eqn:dFlambda}--\eqref{eqn:Kdef} show that,
if $u \in H^1_\pi (\RR^{n+k})$ achieves the minimum of $\II (\beta, \pi )$
defined by
\begin{equation}
\label{eqn:IIdefaxial}
\II (\beta, \pi) = \inf \{ F_\lambda (u) \ : \ u \in H^1_\pi ( \RR^{n+k}) ,
\ J_p (u) = \beta \},
\end{equation}
then
\begin{equation}
\label{eqn:axialellip}
-\Delta u + \lambda u = K |u|^{p-1} u,
\end{equation}
with 
\begin{equation}
\label{eqn:Kdefaxial}
K = \beta^{-1} \II (\beta, \pi).
\end{equation}
Note that
\begin{equation}
\label{eqn:IIlb}
\II (\beta, \pi) \geq \II (\beta),
\end{equation}
where 
\begin{equation}
\label{eqn:IIH1}
\II (\beta) = \{ F_\lambda (u) \ : \ u \in H^1 ( \RR^{n+k}), \ J_p (u) =
\beta \}.
\end{equation}
It follows from \eqref{eqn:axialsobem} that $\| u \|_{L^{p+1}}^2 \leq C
F_\lambda (u)$ for some $C \in (0, \infty)$, which in turn implies $\II
(\beta) > 0$, so also is $\II (\beta, \pi) > 0$.  Parallel to \eqref{eqn:ua},
we can multiply \eqref{eqn:axialellip} by a constant to obtain a solution to
\begin{equation}
\label{eqn:nlbsaxial}
\Delta u - \lambda u + |u|^{p-1} u = 0.
\end{equation}

In Section \ref{sec:vortaxialFlambda} we establish the existence of
$F_\lambda$-minimizers, producing axial vortex solutions to
\eqref{eqn:nlbsaxial}, given natural hypotheses on $p$ and $\lambda$
(cf.~\eqref{eqn:axialpbds}).  Analogous arguments can be brought to bear
to produce axial vortices that are energy minimizers or Weinstein
functional maximizers, but we do not pursue the details here.

Parallel to the ODE study in Section \ref{sec:vortspherODE}, Section
\ref{sec:vortaxialPDE} derives reduced variable PDE for axial standing
waves, and uses these equations to derive further information about these
solutions.

\subsection{$F_\lambda$-minimizers}
\label{sec:vortaxialFlambda}

Here we tackle minimization of $F_\lambda (u)$ over $u \in H^1_\pi (\RR^{n+k})$,
subject to the constraint $J_p (u) = \beta$ under hypotheses
\eqref{eqn:axialpbds} on $p$ and $\lambda$.  Our argument follows one given
in Section $2.1$ of \cite{CMMT}, with some necessary differences in detail.
Thus, with $\II( \beta, \pi)$ as in
\eqref{eqn:IIdefaxial}, take $u_\nu \in H^1_\pi(\RR^{n+k})$ such that
\begin{equation}
\label{eqn:axialminseq}
J_p (u_\nu) = \beta, \ F_\lambda (u_\nu) \leq \II (\beta, \pi ) +
\frac{1}{ \nu}.
\end{equation}
Passing to a subsequence if necessary, we have
\begin{equation}
\label{eqn:wkconvaxial}
u_\nu \to u \in H^1_\pi (\RR^{n+k}),
\end{equation}
converging in the weak topology.  Rellich's theorem implies
\begin{equation}
\label{eqn:H1embaxial}
H^1 (\RR^{n+k}) \hookrightarrow L^{p+1} (\Omega)
\end{equation}
is compact provided $\Omega \subset \RR^{n+k}$ is relatively compact.  Thus,
for such $\Omega$,
\begin{equation}
\label{eqn:Lpconvaxial}
u_\nu \to u \ \text{in} \ L^{p+1} (\Omega)\text{-norm}.
\end{equation}

To proceed, we use the concentration-compactness method, previewed in
Section \ref{sec:vortspherenmin}.  In particular, we will use Lemma
\ref{lem:conccompman} with
\begin{equation}
\label{eqn:ccmeasaxial}
\mu_\nu (E) = \int_E | u_\nu |^{p+1}\, d \Vol.
\end{equation}
This lemma provides a vanishing-concentration-splitting trichotomy, and we
must show that concentration is the only possibility.  First we show that
vanishing cannot occur. One tool will be Lemma \ref{lem:van}, which
we restate here in a slightly different form, since the dimension count has
changed.

\begin{lemma}
\label{lem:axialnovanishing}
Assume $\{ u_\nu \}$ is bounded in $H^1 (\RR^{n+k})$ and
\begin{equation}
\label{eqn:axialvan}
\lim_{\nu \to \infty} \sup_{z \in \RR^{n+k}} \int_{B_R(z)} | u_\nu |^2 d \Vol
= 0, \ \text{for some} \ R > 0.
\end{equation}
Then,
\begin{equation}
\label{eqn:Lrto0axial}
2<r<\frac{2 (n+k)}{n+k-2} \Longrightarrow \| u_\nu \|_{L^r (\RR^{n+k})} \to 0.
\end{equation}
\end{lemma}

\begin{corollary}
Suppose $\{ u_\nu \}$ satisfies \eqref{eqn:axialminseq}.  Then, no
subsequence can satisfy the vanishing condition \eqref{eqn:ccvan}, with
$\mu_\nu$ as in \eqref{eqn:ccmeasaxial}.
\end{corollary}

\begin{proof}
Assume vanishing as in \eqref{eqn:ccvan} does not occur.  Then, by H\"older's
inequality on finite measure balls, \eqref{eqn:axialvan} holds.  Then
\eqref{eqn:Lrto0axial} holds with $r = p+1$.  This contradicts the assumption
$J_p (u) = \beta > 0$.
\end{proof}

To show splitting is impossible, we note that $\II (\beta, \pi)$ has the
following property.  For all $\beta > 0$,
\begin{equation}
\label{eqn:subaddaxial}
\II (\beta, \pi) < \II (\eta, \pi) + \II (\beta - \eta, \pi),
\quad \forall\, \eta \in (0, \beta).
\end{equation}
This follows immediately from 
\begin{equation}
\label{eqn:IIscalingaxial}
\II (\beta, \pi) = \II (1, \pi) \beta^{{2}/({p+1})} .
\end{equation}
The identity \eqref{eqn:IIscalingaxial} follows via a computation identical
to (2.1.8)--(2.1.12) in \cite{CMMT}.

We now show that the splitting cannot occur.

\begin{lemma}
\label{lem:nosplitaxial}
In the setting of \eqref{eqn:subaddaxial}, if $\{ u_\nu \} \subset H^1_\pi
(\RR^{n+k})$ is an $F_\lambda$ minimizing sequence as in \eqref{eqn:axialminseq},
then splitting as in \eqref{eqn:ccsplit1}--\eqref{eqn:ccsplit2} with $\mu_\nu$
as in \eqref{eqn:ccmeasaxial} cannot occur.
\end{lemma}

\begin{proof}

Assume splitting does occur. In other words,
there exists $\alpha\in (0,\beta)$ and for each $\epsilon>0$, sets
$E^\sharp_\nu,\ E^b_\nu\subset\RR^{n+k}$ such that
\eqref{eqn:ccsplit1}--\eqref{eqn:ccsplit2} occur,
with $\mu_\nu$ as in \eqref{eqn:ccmeasaxial}.
We can assume that $E_\nu^\sharp$ and $E_\nu^b$ are
invariant under the action of $G$.  Choose $\epsilon > 0$ sufficiently small
such that
\begin{equation}
\label{eqn:IIubepsaxial}
\II (\beta, \pi) < \II (\alpha, \pi) + \II ( \beta - \alpha, \pi) -
C_1 \epsilon,
\end{equation}
where $C_1 > 0$ is a sufficiently large constant to be fixed later.
Since $\| u_\nu \|_{H^1 (M)}$ and $\| u_\nu \|_{L^{p+1} (M)}$ are uniformly
bounded, it follows from \eqref{eqn:ccsplit1} that there exists $\nu_1$ such
that $\nu \geq \nu_1$ implies
\begin{equation}
\label{eqn:boundaxial}
\int_{S_\nu} |u_\nu|^2\, d \Vol < \epsilon,
\end{equation}
where $S_\nu$ is a set of the form
\begin{equation}
\label{eqn:Snuaxial}
S_\nu = \{ z \in \RR^{n+k} \ : \ d_\nu < \dist (z, E_\nu^\sharp) \leq d_\nu
+ 2 \} \subset \RR^{n+k} \setminus ( E_\nu^\sharp \cup E_\nu^b),
\end{equation}
for some $d_\nu > 0$.  In other words, for $r > 0$, $\nu \geq \nu_1$ we have
\begin{equation}
S_\nu = \tilde{E}_\nu (d_\nu + 2) \setminus \tilde{E}_\nu (d_\nu),
\end{equation}
where
\begin{equation}
\label{eqn:tildeEnuaxial}
\tilde{E}_\nu (r) = \{ z \in \RR^{n+k} \ : \ \dist (z, E_\nu^\sharp) \leq r \}.
\end{equation}
Now, define functions $\chi_\nu^\sharp$ and $\chi_\nu^b$ by
\begin{align}
\chi_\nu^\sharp & = 1, \hspace{3cm} \text{if} \ x \in \tilde{E}_\nu (d_\nu),
\notag \\
 & = 1- \dist (x, \tilde{E}_\nu (d_\nu) ), \quad \text{if} \ x \in
 \tilde{E}_\nu (d_\nu + 1), \label{eqn:chinusharpaxial} \\
 & = 0, \hspace{3cm} \text{if} \ x \notin \tilde{E}_\nu (d_\nu + 2), \notag 
 \end{align}
 and 
\begin{align}
\chi_\nu^b & = 0, \hspace{3cm} \text{if} \ x \in \tilde{E}_\nu (d_\nu + 1),
\notag \\
 & = \dist (x, \tilde{E}_\nu (d_\nu+1) ),
 \quad \text{if} \ x \in \tilde{E}_\nu
 (d_\nu + 2), \label{eqn:chinubaxial} \\
 & = 1, \hspace{3cm} \text{if} \ x \notin \tilde{E}_\nu (d_\nu + 2). \notag 
 \end{align}
 These functions are both Lipschitz with Lipschitz constant $1$ and almost
 disjoint supports.  Also, they are invariant under the action of $G$.  Set
 \begin{equation}
 \label{eqn:unusharp_baxial}
 u_\nu^\sharp = \chi_\nu^\sharp u_\nu,\ \ u_\nu^b = \chi_\nu^b u_\nu \in
 H^1_\pi (\RR^{n+k}).
 \end{equation}
 Note that since $0 \leq \chi_\nu^\sharp + \chi_\nu^b \leq 1$, we have
 \begin{equation}
 \label{eqn:Jpaxial}
 J_p ( u_\nu^\sharp) + J_p (u_\nu^b) = \int (\chi_\nu^\sharp + \chi_\nu^b)
 | u_\nu|^{p+1} d\Vol \leq J_p (u_\nu) = \beta.
 \end{equation}
 Also, given $\lambda > 0$,
 \begin{equation}
 \label{eqn:lambdaaxialbds}
\lambda \| u_\nu^\sharp \|_{L^2}^2 + \lambda \| u_\nu^b \|_{L^2}^2 \leq
\lambda \| u_\nu \|_{L^2}^2.
 \end{equation}
 We have $\nabla u_\nu^\sharp = \chi_\nu^\sharp \nabla u_\nu + (\nabla
 \chi_\nu^\sharp) u_\nu$, and similarly for $u_\nu^b$, and $| \nabla
 \chi_\nu^\sharp| \leq 1$ except for a set of measure $0$, so
 \begin{equation}
 \label{eqn:nablalambdabdaxial}
 \| \nabla  u_\nu^\sharp \|_{L^2}^2 +  \| \nabla u_\nu^b \|_{L^2}^2 \leq
 \| \nabla u_\nu \|_{L^2}^2 + \int_{S_\nu} |u_\nu|^2 d \Vol.
 \end{equation}
 Hence, 
 \begin{equation}
 \label{eqn:Fsplitbdsaxial}
F_\lambda (u_\nu^\sharp) + F_\lambda (u_\nu^b) \leq F_\lambda (u_\nu) +
\epsilon.
 \end{equation}
 Using the support properties of $u_\nu^\sharp$, $u_\nu^b$, together with
 \eqref{eqn:ccsplit1}--\eqref{eqn:ccsplit2} yields
 \begin{equation}
 \label{eqn:Jpbdsaxial}
 | J_p ( u_\nu^\sharp) - \alpha|, \ | J_p ( u_\nu^b ) - (\beta - \alpha)|
 \leq 3 \epsilon.
 \end{equation}
Combining \eqref{eqn:Fsplitbdsaxial} and \eqref{eqn:Jpbdsaxial} and letting
$\nu \to \infty$,
 \begin{equation}
 \label{eqn:IIlbepsaxial}
 \II (\alpha, \pi) + \II (\beta-\alpha, \pi)  \leq \II ( \beta, \pi) + C_0
 \epsilon.
 \end{equation}
 Hence, if $C_1$ is chosen sufficiently large in \eqref{eqn:IIubepsaxial}
 (which simply amounts to producing $\epsilon>0$ sufficiently small), we
 contradict \eqref{eqn:IIubepsaxial}.  This contradiction proves Lemma
 \ref{lem:nosplitaxial}.
 
\end{proof}

Using Lemmas \ref{lem:axialnovanishing} and \ref{lem:nosplitaxial}, we have
the following proposition, which states that for a minimizing sequence
$(u_\nu)$, only the concentration phenomenon can occur.

\begin{proposition}
\label{prop:concaxial}
Let $\{ u_\nu \} \subset H^1_\pi ( \RR^{n+k})$ be a minimizing sequence, as
in \eqref{eqn:axialminseq}.  Then, every subsequence of $\{ u_\nu \}$ has a
further subsequence (which we continue to denote $\{ u_\nu \}$) with the
following property.  There exists a subsequence $\{ z_\nu \} \subset
\RR^{n+k}$ and a function $\tilde{R} (\epsilon)$ such that for all $\nu$,
\begin{equation}
\label{eqn:concaxial}
\int_{B_{\tilde{R} (\epsilon)} (z_\nu)} | u_\nu |^{p+1}\, d \Vol > \beta -
\epsilon, \ \forall \epsilon > 0.
\end{equation}
\end{proposition}

\begin{remark}
The sequence $\{ z_\nu \}$ is independent of $\epsilon > 0$ and the function
$\tilde{R} (\epsilon)$ is independent of $\epsilon$.
\end{remark}

Proposition \ref{prop:concaxial} is about concentration along subsequences of
a minimizing sequence.  We use the group of $\RR^k$-translations to proceed
from there to compactness result.  In more detail, we have a sequence
$\{ u_\nu \} \subset H^1_\pi (\RR^{n+k})$ satisfying \eqref{eqn:axialminseq}.
After passing to a subsequence if necessary, Proposition \ref{prop:concaxial}
shows that we have points $z_\nu \in \RR^{n+k}$ and a function $\tilde{R}
(\epsilon)$ such that \eqref{eqn:concaxial} holds.  Now $z_\nu =
(x_\nu, y_\nu)$ with $x_\nu \in \RR^n$, $y_\nu \in \RR^k$, and
$y$-translation preserves $H^1_\pi (\RR^{n+k})$, so we can assume each
$y_\nu = 0$.  On the other hand, since $u_\nu \in H^1_\pi (\RR^{n+k})$,
\eqref{eqn:concaxial} implies a bound $| x_\nu| \leq K < \infty$, for all
$\nu$.

With these adjustments made, we now have
\begin{equation}
\label{eqn:conaxial1}
\int_{B_{\tilde{R} (\epsilon) + K} (0)} | u_\nu |^{p+1}\, d \Vol > \beta -
\epsilon, \quad \forall\, \epsilon > 0.
\end{equation}
as well as \eqref{eqn:wkconvaxial} and \eqref{eqn:Lpconvaxial}.  Hence,
\begin{equation}
\label{eqn:strconvaxial}
u_\nu \to u \ \text{in} \ L^{p+1} (\RR^{n+k})\text{-norm, and}
\ J_p (u) = \beta.
\end{equation}
Since $F_\lambda$ is comparable to the $H^1$-norm squared, we have
\begin{equation}
\label{eqn:H1eqFlambdaaxial}
F_\lambda (u) \leq \liminf_{ \nu \to \infty} F_\lambda (u_\nu) = \II (\beta,
\pi).
\end{equation}
Given \eqref{eqn:strconvaxial}, we have 
\begin{equation}
\label{eqn:H1eqFlambdaaxial1}
 F_\lambda (u) = \II (\beta, \pi).
\end{equation}
As a result, the following conclusion holds.

\begin{proposition}
\label{prop:axialFlammin}
In the setting of Proposition \ref{prop:concaxial}, $\{ u_\nu \}$ has a
subsequence satisfying \eqref{eqn:strconvaxial} with limit $u \in H^1_\pi
(\RR^{n+k})$, the desired $F_\lambda$ minimizer.
\end{proposition}

\begin{remark}
It follows from \eqref{eqn:H1eqFlambdaaxial1} that convergence $u_\nu \to u$
in \eqref{eqn:wkconvaxial} holds in norm in $H^1_\pi (\RR^{n+k})$.
\end{remark}

Existence of energy minimizers and of Weinstein functional maximizers,
under appropriate constraints on the exponent $p$, can also be derived,
using a mixture of techniques developed in this section and techniques from
Sections \ref{sec:vortspherenmin}--\ref{sec:vortspherWmax}.
We leave the details to the interested reader.



\subsection{Reduced PDE for axial standing waves}
\label{sec:vortaxialPDE}

Here we assume $\pi$ is an irreducible representation of $G$ on $V$ and $V_0$
has the property
\begin{equation}
\label{eqn:dimV0axial}
\dim V_0 = 1.
\end{equation}
Then, an axial vortex $u \in H^1_\pi (\RR^{n+k})$ has the form
\begin{equation}
\label{eqn:urgp0axial}
u(r g \cdot p_0, y) = \psi (r,y) \pi (g)  \phi,
\end{equation}
with $\phi \in V_0$ a unit spanning vector and $\psi : \RR^+ \times \RR^k
\to \CC$.  If $u$ solves \eqref{eqn:nlbsaxial}, then parallel to
calculations in Section \ref{sec:vortspherODE}, we obtain for $\psi$ the PDE
\begin{equation}
\label{eqn:psiPDE}
\pa_r^2 \psi + \Delta_y \psi + \frac{n-1}{r} \pa_r \psi -
\frac{ \mu_\pi^2}{r^2} \psi - \lambda \psi + | \psi |^{p-1} \psi = 0,
\end{equation}
with $\mu_\pi^2 \in \RR^+$ as in \eqref{eqn:SSspec}.  We can write
\eqref{eqn:urgp0axial} as
\begin{equation}
\label{eqn:uxyaxial}
u(x,y) = u_0 (x,y) \pi (g)  \phi, \quad x = r g \cdot p_0,
\end{equation}
with 
\begin{equation}
\label{eqn:psiaxial}
u_0 (x,y) = \psi (|x|, y),
\end{equation}
and hence \eqref{eqn:psiPDE} is also equivalent to
\begin{equation}
\label{eqn:u0PDE}
\Delta u_0 - \lambda u_0 - \frac{ \mu_\pi^2}{|x|^2} u_0 + |u_0|^{p-1} u_0 = 0,
\end{equation}
for the scalar function $u_0$.  

Parallel to \eqref{eqn:sharpmap}, we can define a map $u \to u^\sharp$
on $H^1_\pi (\RR^{n+k})$:
\begin{equation}
\label{eqn:sharpmapaxial}
u^\sharp (r g \cdot p_0, y) = | \psi (r,y) | \pi (g) \phi,
\end{equation}
for $u$ as in \eqref{eqn:urgp0axial}.  We have for a.e. $z \in \RR^{n+k}$, 
\begin{equation}
\label{eqn:usharpbdsaxial}
| u^\sharp (z ) | = | u(z) |, \quad
| \nabla u^\sharp (z) | \leq | \nabla u (z) |.
\end{equation}
Hence,
\begin{equation}
\label{eqn:usharpbdsaxial1}
\| u^\sharp \|_{L^2} = \| u \|_{L^2}, \quad
\| u^\sharp \|_{L^{p+1}} = \| u \|_{L^{p+1}}, \quad
\| \nabla u^\sharp \|_{L^2} \leq \| \nabla u \|_{L^2}.
\end{equation}
Consequently, if $u$ is an $F_\lambda$-minimizer within $H^1_\pi (\RR^{n+k})$,
or an energy minimizer or Weinstein maximizer, so is $u^\sharp$.  This forces
$\| \nabla u^\sharp \|_{L^2} = \| \nabla u \|_{L^2}$, and hence
\begin{equation}
\label{eqn:nablasharpbdaxial}
| \nabla u^\sharp (z) | = | \nabla u (z) |,
\end{equation}
for a.e. $z \in \RR^{n+k}$.  We also have for
\begin{equation}
\label{eqn:u0sharpaxial}
u_0^\sharp (x,y) = | \psi (|x|, y)| = | u_0 (x,y)| = |u(x,y)|
\end{equation}
that
\begin{equation}
\label{eqn:u0PDEaxial}
\Delta u_0^\sharp - \lambda u_0^\sharp - \frac{\mu_\pi^2}{|x|^2} u_0^\sharp +
( u_0^\sharp )^p = 0,
\end{equation}
and $u_0^\sharp \geq 0$ on $\RR^{n+k} \setminus \{ x = 0 \}$.  Harnack's
inequality (see \cite{GiTr-book}, Theorem 8.20) then implies
\begin{equation}
\label{eqn:u0sharpharnackaxial}
u_0^\sharp (x,y) > 0, \quad \text{if} \ x \neq 0.
\end{equation}
Consequently,
\begin{equation}
\label{eqn:unon0axial}
u(x,y) \neq 0, \quad \text{if} \ x \neq 0.
\end{equation}

Returning to \eqref{eqn:u0sharpharnackaxial}, note that parallel to
\eqref{eqn:kinsphdecomp},
\begin{equation}
\label{eqn:kinsphdecompaxial}
\| \nabla u \|_{L^2}^2 = \| \nabla u_0 \|_{L^2}^2
+ \mu_\pi^2 \Bigl\|\frac{ u_0}{|x|} \Bigr\|_{L^2}^2.
\end{equation}
Similarly, we have
\begin{equation}
\label{eqn:kinsphdecompaxial1}
\| \nabla u^\sharp \|_{L^2}^2 = \| \nabla u_0^\sharp \|_{L^2}^2
+ \mu_\pi^2 \Bigl\|\frac{ u_0^\sharp }{|x|} \Bigr\|_{L^2}^2.
\end{equation}
Since $|u_0^\sharp (z) | = | u_0 (z) |$, and the left sides of
\eqref{eqn:kinsphdecompaxial} and \eqref{eqn:kinsphdecompaxial1} have been
seen to be equal, we deduce that $\| \nabla u_0 \|_{L^2}^2 = \| \nabla
u_0^\sharp \|_{L^2}^2$, and since $| \nabla u_0| \leq | \nabla u_0^\sharp
|$ a.e. on $\RR^{n+k}$, we hence have
\begin{equation}
\label{eqn:u0eqaxial} 
| \nabla u_0 (x,y) | = | \nabla u_0^\sharp (x,y) |, \quad \text{a.e. on} \
\RR^{n+k},
\end{equation}
or equivalently 
\begin{equation}
\label{eqn:psieqaxial} 
| \nabla \psi (r,y) | = | \nabla |\psi| (r,y) |, \quad \text{a.e. on} \ \RR^+
\times \RR^{k}.
\end{equation}
Writing
\begin{equation}
\label{eqn:psidecompaxial}
\psi = \omega | \psi|, \quad \omega : \RR^{n+k} \setminus \{ x = 0 \} \to S^1
\subset \CC,
\end{equation}
we see from \eqref{eqn:psieqaxial}  that $\omega$ is constant.  We hence have
the following extension of Proposition \ref{prop:psireal}.

\begin{proposition}
\label{prop:ODEaxial}
Let $u \in H^1_\pi (\RR^{n+k})$ be an $F_\lambda$-minimizer (or an energy
minimizer or a Weinstein functional maximizer).  Assume \eqref{eqn:dimV0axial}
holds.  Then, after multiplication by a complex constant, $u$ has the form
\eqref{eqn:urgp0axial} with
\begin{equation}
\label{eqn:axialpsimap}
\psi : \RR^{n+k}  \setminus \{ x = 0 \} \to (0, \infty).
\end{equation}
\end{proposition}

\section{Mass Critical NLS with vortex initial data}
\label{sec:masscrit}

Here we discuss solutions to the mass critical NLS (see \eqref{eqn:NLSmc}
below), with vortex initial data (sometimes, with initial data of a more
general nature).  One result is that if the initial data $v_0$ belongs to
$H^1_\pi(\RR^n)$ and has mass $\|v_0\|^2_{L^2}$ less than that of the
corresponding Weinstein functional maximizer, then the solution exists
for all $t$.  This extends results of \cite{FG}, which deal with $v_0\in
H^1_\ell(\RR^2)$.  In fact, in Section \ref{s4.1}, we treat more general
$n$-dimensional Riemannian manifolds $M$, in cases where there is no
Weinstein functional maximizer.  Then the results have a more subtle
formulation.  See Propositions \ref{prop:NLSmcge} and \ref{p4.1.3}.
Section \ref{s4.2} gives some monotonicity results, regarding how the
Weinstein functional supremum $\mathcal{W}(\pi)$ depends on $\pi$, expanding
on some special cases from Section \ref{sec:vortspherWmax}.
Section \ref{sec:vortscat} investigates scattering of solutions to mass
critical NLS with vortex initial data.

\subsection{General global existence result}
\label{s4.1}

We begin with a general global existence result. Let us consider the
mass-critical, focusing, nonlinear Schr\"odinger equation
\begin{equation}
\label{eqn:NLSmc}
i \pa_t v + \Delta v + |v |^{p-1} v = 0, \quad p = 1 + \frac{4}{n},
\end{equation}
on $I \times M$, where $M$ is an $n$-dimensional Riemannian manifold,
possibly with boundary.  We take initial data
\begin{equation}
\label{eqn:NLSmcid}
v(0) = v_0 \in H^1 (M)\quad (\text{or }\ H^1_0(M)).
\end{equation}
We make the following
\newline
{\bf Hypothesis.} The initial value problem
\eqref{eqn:NLSmc}--\eqref{eqn:NLSmcid} is locally well posed.
In particular, it has a
unique solution in $C(I,H^1(M))$ on a time interval $I$ that is a function of
$\|v_0\|_{H^1}$.

\begin{remark} 
\label{rem:H1pilwp}
We pause to mention some cases in which this hypothesis is known to hold.
First is the classical case $M=\RR^n$, where local well posedness for
$v_0\in H^1$ was established by \cite{GiVe} for the NLS \eqref{eqn:nls}
whenever $1<p<1+4/(n-2)$.  A proof of this can be found on pp.~93--98 of
\cite{Caz_book}.  (Further work treated also the energy-critical exponent
$p=1+4/(n-2)$, for $n\ge 3$.)  As noted by \cite{Ban} (p.~1647), the classical
Strichartz estimates used in this proof apply also to $M=\HH^n$, hyperbolic
space, and one also has such local well posedness of \eqref{eqn:nls} with
$v_0\in H^1(\HH^n)$.  See also \cite{AP}.  Results in case
\beq
M=\RR^n\setminus K,
\label{4.1.K}
\eeq
when $K$ is a compact, smoothly bounded, strongly convex obstacle, and the
Dirichlet boundary condition is given on $\pa K$, are obtained in \cite{Iv}.
It is shown that all the classical Strichartz estimates, except for the
endpoint case, continue to hold.  Again, the arguments given on pp.~93--98
of \cite{Caz_book} can be used to establish well posedness of \eqref{eqn:nls}
for initial data in $H^1_0(M)$, whenever $1<p<1+4/(n-2)$.  For application
to our vortex setting, note that we can take $K=B$, a ball in $\RR^n$.
(In \cite{Iv}, the energy critical case $n=3,p=5$ is also treated;
see \cite{KiViZh} for a treatment in higher dimensions.)
Going further, \cite{BuGeTz} considers
$M$ as in \eqref{4.1.K} for smoothly bounded, nontrapping obstacles $K\subset
\RR^n$.  These authors derived Strichartz estimates, with a loss.  This led to
local well posedness for initial data in $H^1_0(M)$, for a more restricted
class of NLS equations.  However, as shown on pp.~309--312 of \cite{BuGeTz},
their main Strichartz estimate does yield well posedness, with initial data
in $H^1_0(M)$, for such $M$, for a class of equations that contains
\eqref{eqn:NLSmc} in dimensions $n=2$ and $n=3$.
\end{remark}

Returning to our main line of inquiry, we have, for the local solution to
\eqref{eqn:NLSmc}--\eqref{eqn:NLSmcid},
conservation of mass
\begin{equation}
\label{eqn:NLSmcMass}
Q( v(t) )  = \| v(t) \|_{L^2}^2,
\end{equation}
and of energy
\begin{equation}
\label{eqn:NLSmcEnergy}
E( v(t) )  = \frac12 \| \nabla v(t) \|_{L^2}^2 - \frac{1}{p+1} \int_M
|v(t)|^{p+1}\, d \Vol.
\end{equation}
We will assume there is a Gagliardo-Nirenberg estimate
\begin{equation}
\label{eqn:NLSmcGN}
\| u \|_{L^{p+1}} \leq C_M \| u \|_{L^2}^{1-\gamma} \| \nabla u
\|_{L^2}^\gamma, \quad \forall\, u \in H^1 (M),
\end{equation}
where, with $p$ as in \eqref{eqn:NLSmc},
\begin{equation}
\label{eqn:NLSmcGammadef}
\gamma = \frac{n}{2} - \frac{n}{p+1} = \frac{2}{p+1} = \frac{n}{n+2}.
\end{equation}
It is classical that \eqref{eqn:NLSmcGN} holds if $M = \RR^n$.
It also holds for $M = \HH^n$ and when $M$ is a compact and connected,
with non-empty boundary (and one insists $u \in H^1_0 (M))$.  The classic
paper \cite{W} shows that for $M = \RR^n$ if
\begin{equation}
\label{eqn:Weinstein}
v_0 \in H^1 (M) \ \text{and}  \ \| v_0 \|_{L^2} < \left(
\frac{p+1}{2C_M^{p+1}} \right)^{{1}/({p-1})},
\end{equation}
then \eqref{eqn:NLSmc}--\eqref{eqn:NLSmcid} has a global solution.
Here we note a related result, generalizing both the global existence result
just mentioned and that for planar vortex initial data in \cite{FG}.

Here is our setting.  Take
\begin{equation}
\label{eqn:vortsubspace}
\calH \subset H^1(M,V), \ \text{a closed linear subspace},
\end{equation}
where $V$ is a finite dimensional inner product space.  Assume that if $v_0
\in \calH$, then the short time solution to
\eqref{eqn:NLSmc}--\eqref{eqn:NLSmcid} has the property that $v(t) \in \calH$.
We move from \eqref{eqn:NLSmcGN} to the hypothesis
\begin{equation}
\label{eqn:NLSmcGNcalH}
\| u \|_{L^{p+1}} \leq C_{\calH} \| u \|_{L^2}^{1-\gamma} \| \nabla u
\|_{L^2}^\gamma, \quad \forall\, u \in \calH,
\end{equation}
with $p$ as in \eqref{eqn:NLSmc} and $\gamma$ as in \eqref{eqn:NLSmcGammadef},
noting that we might well have
\begin{equation}
\label{eqn:CHltCM}
C_\calH < C_M.
\end{equation}
Here is the global existence result.

\begin{proposition}
\label{prop:NLSmcge}
With $p$, $\gamma$ and $\calH$ as above, assume \eqref{eqn:NLSmcGNcalH} holds.
If $v_0 \in \calH$ satisfies
\begin{equation}
\label{eqn:NLSmcidubd}
\| v_0 \|_{L^2} < \left( \frac{p+1}{2 C_\calH^{p+1} } \right)^{{1}/({p-1})},
\end{equation}
then \eqref{eqn:NLSmc}--\eqref{eqn:NLSmcid} has a solution for all $t \in \RR$.
\end{proposition}

\begin{proof}
We follow Weinstein's classic argument from \cite{W}.  If $ v \in C(I, H^1(M,V))$
solves \eqref{eqn:NLSmc}, it suffices to control $\| \nabla v \|_{L^2}^2$.
The stated hypotheses give $v \in C (I, \calH)$, and then
\eqref{eqn:NLSmcEnergy} and \eqref{eqn:NLSmcGNcalH} give
\beq
\aligned
\frac12 \| \nabla v (t) \|_{L^2}^2 & = E (v(t)) + \frac{1}{p+1} \| v(t)
\|_{L^{p+1}}^{p+1}  \\
& \leq E(v(t)) + \frac{ C_\calH^{p+1}}{p+1} \| v(t) \|_{L^2}^{p-1} \| \nabla v(t)
\|_{L^2}^2  \\
& =E(v_0) + \frac{ C_\calH^{p+1}}{p+1} \| v_0 \|_{L^2}^{p-1} \| \nabla v(t)
\|_{L^2}^2  . 
\endaligned
\label{eqn:NLSmcgradeqE}
\eeq
Now, \eqref{eqn:NLSmcidubd} is equivalent to
\begin{equation}
\label{eqn:NLSmcL2bdcon}
\frac{C_\calH^{p+1}}{p+1} \| v_0 \|_{L^2}^{p-1} = \sigma < \frac12,
\end{equation}
which gives
\begin{equation}
\label{eqn:NLSmcgradbdE}
\left( \frac12 - \sigma \right) \| \nabla v (t) \|_{L^2}^2 \leq E( v_0 ),
\end{equation}
and hence the desired upper bound holds on $\| \nabla v (t) \|_{L^2}^2$.  
\end{proof}

Still following \cite{W}, we relate the material above to the behavior of the
Weinstein functional
\begin{equation}
\label{eqn:NLSmcW}
W (u) = \frac{  \| u \|_{L^{p+1}}^{p+1} }{ \| u \|_{L^2}^\alpha \| \nabla u
\|_{L^2}^\beta}, \ \ \alpha = p-1, \ \beta = 2,
\end{equation}
with $p$ as in \eqref{eqn:NLSmc}.  Note that \eqref{eqn:NLSmcGNcalH} holds with
\begin{equation}
\label{eqn:CdefW}
C_\calH^{p+1}  = \sup \{ W(u) \ : \ 0 \neq u \in \calH \}.
\end{equation}
As shown in \cite{W}, when $\calH = H^1 (\RR^n)$, and in this paper for
certain cases $\calH = H^1_\pi (\RR^n)$, there are cases when $W(u)$ can
achieve a maximum.  As in \eqref{eqn:dWdtau}--\eqref{eqn:lamKdefW}, such a
maximizer solves
\begin{equation}
\label{eqn:NLSmcEllip}
\Delta u - \lambda u + K |u|^{p-1} u = 0, 
\end{equation}
with
\begin{equation}
\label{eqn:NLSmcKdef}
K = \frac{p+1}{\beta} \frac{ \| \nabla u \|_{L^2}^2}{ \| u
\|_{L^{p+1}}^{p+1}}.
\end{equation}
Multiplying $u$ by a constant achieves $K =1$, and as a result $v(t,x) =
e^{i \lambda t} u(x)$ is a standing wave solution to \eqref{eqn:NLSmc}.
We call such a $u$ a ground state.  In such a case, we have
\begin{equation}
\label{eqn:NLSmcLpH1eq}
\| u \|_{L^{p+1}}^{p+1} = \frac{p+1}{2} \| \nabla u \|_{L^2}^2.
\end{equation}
Comparison with $W(u) = C_\calH^{p+1}$, or equivalently
\begin{equation}
\label{eqn:NLSmcshaprGNcalH}
\| u \|_{L^{p+1}}^{p+1} = C_\calH^{p+1} \| u \|_{L^2}^{p-1} \| \nabla u
\|_{L^2}^2,
\end{equation}
gives
\begin{equation}
\label{eqn:NLSmcuL2eqpcon}
\| u \|_{L^2} = \left( \frac{p+1}{2 C_\calH^{p+1}} \right)^{{1}/({p-1})}.
\end{equation}

Thus, we have the following consequence of Proposition \ref{prop:NLSmcge}.

\begin{proposition}
\label{prop:NLSmcge1}
In the setting of Proposition \ref{prop:NLSmcge}, assume
$u_{\Cal{H}} \in \calH$
maximizes $W(u)$ in \eqref{eqn:CdefW} and normalize $u_{\Cal{H}}$ to solve
\eqref{eqn:NLSmcEllip} with $K = 1$.  If $v_0 \in \calH$ satisfies
\begin{equation}
\label{eqn:NLSmcL2bdid}
\| v_0 \|_{L^2} < \| u_{\Cal{H}} \|_{L^2},
\end{equation}
then \eqref{eqn:NLSmc}--\eqref{eqn:NLSmcid} has a solution for all $t \in \RR$.
\end{proposition}

\begin{remark}
As noted in Section 4.3 of \cite{CMMT}, for many manifolds $M$ there are no
Weinstein functional maximizers.  Hence, there are many settings where
Proposition \ref{prop:NLSmcge} applies but Proposition \ref{prop:NLSmcge1}
does not.
\end{remark}

\begin{remark}\label{r4.1.5}
We set
\beq
u_{\Cal{H}}=Q_\pi\ \text{ if }\ \Cal{H}=H^1_\pi(\Rr^n).
\label{4.1.24A}
\eeq
Note that
\beq
\Cal{H}=H^1_\pi(\Rr^n)\Longrightarrow \Cal{C}_{\Cal{H}}^{p+1}=\Cal{W}(\pi),
\eeq
so
\beq
\|Q_\pi\|_{L^2}=\Bigl(\frac{p+1}{2\Cal{W}(\pi)}\Bigr)^{1/(p-1)}.
\eeq
\end{remark}

$\text{}$

It is perhaps illuminating to complement
Propositions \ref{prop:NLSmcge}--\ref{prop:NLSmcge1} with the
following result.  Assume $\Cal{H}\subset H^1(M)$ and
$u_{\Cal{H}}\in \Cal{H}$
satisfy the conditions of Proposition \ref{prop:NLSmcge1}.
Let $\widetilde{M}$ be another
complete $n$-dimensional Riemannian manifold and $\wcH\subset H^1(\widetilde{M})$
a closed linear subspace such that if $\tilde{v}_0\in\wcH$ then
\eqref{eqn:NLSmc} has a solution on $I\times\widetilde{M}$, satisfying
\beq
v(0)=\tilde{v}_0,
\label{4.1.24}
\eeq
on an interval whose length depends on $\|\tilde{v}_0\|_{H^1}$, and with
$v(t)\in\wcH$ for $t\in I$.  We continue to take $p=1+4/n$, and define
$W(u)$ as in \eqref{eqn:NLSmcW}.  We define $\Cal{C}_{\wcH}$ by
\beq
\Cal{C}_{\wcH}^{p+1}=\sup \{W(u):0\neq u\in\wcH\}.
\label{4.1.25}
\eeq

\begin{proposition} \label{p4.1.3}
Take $\Cal{H},\ u_{\Cal{H}}\in\Cal{H}$ as in Proposition \ref{prop:NLSmcge1}.
Take $\widetilde{M}$ and $\wcH\subset H^1(\widetilde{M})$ as above.  Assume
\beq
\Cal{C}_{\wcH}=\Cal{C}_{\Cal{H}}.
\label{4.1.26}
\eeq
If $\tilde{v}_0\in\wcH$ satisfies
\beq
\|\tilde{v}_0\|_{L^2(\widetilde{M})}<\|u_{\Cal{H}}\|_{L^2(M)},
\label{4.1.27}
\eeq
then \eqref{eqn:NLSmc}, \eqref{4.1.24} has a solution in $\wcH$
for all $t\in\RR$.
\end{proposition}

\begin{proof}
Proposition \ref{prop:NLSmcge} implies \eqref{eqn:NLSmc}, (\ref{4.1.24})
has a global solution provided
\beq
\|\tilde{v}_0\|_{L^2(\widetilde{M})}<
\Bigl(\frac{p+1}{2\Cal{C}_{\wcH}^{p+1}}\Bigr)^{1/(p-1)}.
\label{4.1.28}
\eeq
Given (\ref{4.1.26}), the conclusion follows from the identity
\eqref{eqn:NLSmcuL2eqpcon}.
\end{proof}

\sk
{\sc Example.} If $B\subset\RR^n$ is a smoothly bounded set, take
\beq
M=\RR^n,\quad \widetilde{M}=\RR^n\setminus B,\quad \Cal{H}=H^1(\RR^n),\quad
\wcH=H^1_0(\RR^n\setminus B).
\label{4.1.29}
\eeq
It is shown in \S{4.3} of \cite{CMMT} that (\ref{4.1.26}) holds,
but there is no $W$-maximizer in $\wcH$.
Hence the condition (\ref{4.1.27}) for global solvability
applies, but this is not a corollary of Proposition \ref{prop:NLSmcge1}.

\begin{remark}
Stimulated by \cite{Dod}, one might consider extending
Proposition \ref{p4.1.3} as follows.  Take $\tilde{v}_0$ in the $L^2$-closure
of $\widetilde{\Cal{H}}$, satisfying \eqref{4.1.27}, and ask for global
solvability.  We do not tackle this here.
\end{remark}

\subsection{Further monotonicity results}
\label{s4.2}

Here we extend the monotonicity results
\eqref{eqn:monpi} and (\ref{W.26}) from the setting
of $M=\RR^n$ (and spherical vortices) to more general settings.  We begin
by extending \eqref{eqn:kinsphdecomp},
working in the setting of spherical vortices.
Thus, we assume $M$ is as in
\eqref{eqn:sphereM}--\eqref{eqn:metric}, and has rotational symmetry.
We also assume
\beq
\text{$\pi$ is irreducible on $V$.}
\label{M.1}
\eeq
In such a case, \eqref{eqn:LapEucAalt} gives
\beq
u\in H^1_\pi(M)\Rightarrow \Delta u=\pa_r^2u+\frac{A'(R)}{A(r)}\pa_r u
-\mu_\pi(r)^2u.
\label{M.2}
\eeq
We also assume
\beq
\dim V_0=1,
\label{M.3}
\eeq
with $V_0$ as in \eqref{eqn:V0}.
Then there exists a unit $\varphi$, spanning $V_0$,
such that
\beq
u(r,g\cdot p_0)=\psi(r)\pi(g)\varphi,\quad \psi:\RR^+\rightarrow \CC,
\label{M.4}
\eeq
and
\beq
u_0(r,g\cdot p_0)=\psi(r)\Longrightarrow u_0\in H^1_r(M),
\label{M.5}
\eeq
with
\beq
\|u\|^2_{L^2}=\|u_0\|^2_{L^2},
\label{M.6}
\eeq
and
\beq
\aligned
\|\nabla u\|^2_{L^2}
&=\|\nabla u_0\|^2_{L^2}-\int_0^\infty (\Delta_{h(r)}u,u) A(r)\, dr \\
&=\|\nabla u_0\|^2_{L^2}+\int_0^\infty \int\limits_{S^{n-1}} |u_0|^2\,dS(\omega)
\, \mu_\pi(r)^2 A(r)\, dr \\
&=\|\nabla u_0\|^2_{L^2}+\int\limits_M |u_0|^2\mu_\pi(r)^2\, dV.
\endaligned
\label{M.7}
\eeq

As noted in \eqref{eqn:hdecomp}--\eqref{eqn:sigmaODE},
if $G=SO(n)$ ($n=\dim M$), then
\beq
h(r)=\sigma(r)\, h_S,
\label{M.8}
\eeq
where $h_S$ is the standard metric on $S^{n-1}$ and $\sigma:I\rightarrow
(0,\infty)$, hence
\beq
\mu_\pi(r)=\frac{\mu_\pi}{\sigma(r)},
\label{M.8A}
\eeq
and we get
\beq
\|\nabla u\|^2_{L^2}=\|\nabla u_0\|^2_{L^2}
+\mu_\pi^2 \Bigl\|\frac{u_0}{\sigma}\Bigr\|^2_{L^2},
\label{M.9}
\eeq
which specializes to \eqref{eqn:kinsphdecomp}
when $M=\RR^n$, in light of \eqref{eqn:hdecomp}.  In such a
case, parallel to \eqref{eqn:Winv}, (\ref{W.24}), and (\ref{W.25}), we have
\beq
\frac{1}{\Cal{W}(\pi)}=\inf\Bigl\{
\frac{\|u_0\|^\alpha (\|\nabla u_0\|^2_{L^2}+\mu_\pi^2\|u_0/\sigma\|^2_{L^2}
)^{\beta/2}}{\|u_0\|_{L^{p+1}}^{p+1}}
:0\neq u_0\in H^1_r(M)\Bigr\},
\label{M.10}
\eeq
\beq
\Cal{I}(\beta,\pi)=\inf \Bigl\{F_\lambda(u_0)+\mu_\pi^2
\Bigl\|\frac{u_0}{\sigma}\Bigr\|^2_{L^2}:u_0\in H^1_r(M),\
J_p(u_0)=\beta\Bigr\},
\label{M.11}
\eeq
and
\beq
\Cal{E}(\beta,\pi)=\inf \Bigl\{E(u_0)+\frac{\mu_\pi^2}{2}
\Bigl\|\frac{u_0}{\sigma}\Bigr\|^2_{L^2}:u_0\in H^1_r(M),\
\|u_0\|^2_{L^2}=\beta\Bigr\},
\label{M.12}
\eeq
under appropriate hypotheses on $p$.  We then have the following.

\begin{proposition} \label{pm.1}
Let $M$ be as in \eqref{eqn:sphereM}--\eqref{eqn:metric}, with rotational
symmetry, and assume (\ref{M.8}) holds.
Let $\pi_1$ and $\pi_2$ be two unitary representations of $G$,
satisfying (\ref{M.1}) and (\ref{M.3}).
Then, under appropriate hypotheses on $p$,
\beq
\aligned
\mu_{\pi_2}^2>\mu_{\pi_1}^2\Longrightarrow\
&\Cal{W}(\pi_2)<\Cal{W}(\pi_1), \\
&\Cal{I}(\beta,\pi_2)>\Cal{I}(\beta,\pi_1), \\
&\Cal{E}(\beta,\pi_2)>\Cal{E}(\beta,\pi_1).
\endaligned
\label{M.13}
\eeq
\end{proposition}


One could tackle analogous results for axial vortices, but we omit this.

\subsection{Mass critical scattering below the vortex mass}
\label{sec:vortscat}

We derive a result on scattering to a linear solution as $t\rightarrow
\pm\infty$ to the $L^2$-critical NLS equation \eqref{eqn:NLSmc} when $M=\RR^n$,
with initial data
\beq
v_0\in \Sigma_\pi=H^1_\pi(\RR^n)\cap H^{0,1}(\RR^n),
\label{4.3.1}
\eeq
where
$$
\|u\|_{H^{0,1}(\RR^n)}=\|\langle x\rangle u\|_{L^2(\RR^n)},
$$
assuming the $L^2$-norm of $v_0$ is below that of $Q_\pi$.

\begin{proposition}
\label{t:scattering}
Let us  take $v_0 \in \Sigma_\pi$ such that
$\| v_0 \|_{L^2 (\mathbb{R}^n)} < \| Q_\pi \|_{L^2 (\mathbb{R}^n)}$,
with $Q_\pi$ defined as in
Proposition \ref{prop:NLSmcge1} and \eqref{4.1.24A},
so that \eqref{eqn:NLSmc}
has a global solution in $H^1_\pi (\mathbb{R}^n)$.  Then, the solution
satisfies $v(t)\in \Sigma_\pi$ for all time,
and there exist $v_{\pm}\in\Sigma_\pi$
such that
\begin{equation}
\lim\limits_{t\rightarrow\pm\infty}\, e^{- it\Delta}v(t)=v_{\pm}.
\end{equation}
\end{proposition}

In outline, our argument follows the pattern presented in Chapter 7 of
\cite{Caz_book}, but with some necessary modifications.

\subsubsection{First $H^{0,1}$ estimates}

We recall Proposition $6.5.1$ from \cite{Caz_book}.

\begin{proposition}
\label{prop:pcwp}
Let $v_0 \in H^1 (\Rr^n)\cap H^{0,1}(\RR^n)$ be the initial data
for \eqref{eqn:NLSmc}.  Given solution $v(x,t) \in C ( (-T_{min}, T_{max}),
H^1 (\Rr^n))$, we have
$$
|x| v(x,t) \in C(( -T_{min}, T_{max}), L^2 (\Rr^n)).
$$
In addition, the map
\begin{equation}
t \to f(t) = \int_{\Rr^n} |x|^2 |v|^2 (x,t) dx \in C^2 (-T_{min}, T_{max})
\end{equation}
with
\begin{equation}
f'(t) = 4 \Im \int_{\Rr^n} \bar{v} x \cdot \nabla v  dx
\end{equation}
and
\begin{equation}
f''(t) = 16 E(v_0),
\label{4.3.B}
\end{equation}
for $E(v_0)$ defined as in \eqref{eqn:energy}.
\end{proposition}

The proof follows from proving uniform bounds in the $\epsilon \to 0$ limit
carefully for the regularized function
\begin{equation*}
f_{\epsilon,m} (t) = \| e^{- \epsilon |x|^2} |x| v_m \|_{L^2 (\Rr^n)}^2
\end{equation*}
assuming that $v_0^{(m)} \in H^2$, then taking a suitable limiting argument
for $\{ v_0^{m} \} \in H^2$ converging to $v_0 \in H^1$.

\subsubsection{The pseudoconformal transformation and conservation law for
mass critical NLS}

The scaling invariance in the mass critical NLS equation 
leads to the conservation law
\begin{equation}
\label{eqn:pscl}
\| (x + 2 i t \nabla) v (t) \|_{L^2}^2 - 4 t^2 \frac{1}{1 + {2}/{n}}
\int |v|^{2 + {4}/{n}} dx = \| x v_0 \|_{L^2}^2 .
\end{equation}
This is essentially stated in Theorem $7.2.1$ from \cite{Caz_book}.  
An equivalent form given in (7.2.8) of \cite{Caz_book}, is that, for
\begin{equation}
\label{eqn:uvpct}
u(t,x) = e^{{- i |x|^2}/{4t}} v(t,x), 
\end{equation}
we have that $u \in H^1$ is well defined and indeed
\begin{equation}
\label{eqn:pctd}
8 t^2  E(u(t)) = \| x v_0 \|_{L^2}^2.
\end{equation}
Another useful variant is
\beq
\|xe^{-it\Delta}v(t)\|_{L^2(\RR^n)}^2-t^2\frac{1}{2+4/n}
\int |v(t,x)|^{2+4/n}\, dx=\|xv_0\|^2_{L^2}.
\label{4.3.3}
\eeq

Furthermore, there is the following pseudoconformal transformation
\begin{equation}
\label{e:pct}
\tilde v (x,t) = t^{-{n}/{2}} e^{ i{|x|^2}/{4t} } v \left(- \frac{1}{t} ,
\frac{x}{t} \right),
\end{equation}
which leaves $\Sigma_\pi$ invariant.
In the case of the $L^2$ critical NLS, \eqref{e:pct} is also a solution to
\eqref{eqn:NLSmc}.
Also, $\| \tilde{v}\|_{L^2} = \| v_0 \|_{L^2}$.
Hence, since $\tilde v(t) \in H^1_\pi$, we are still
guaranteed global existence via Section \ref{s4.1}.

Local well-posedness for $v$
solving \eqref{eqn:NLSmc} in
$H^1$, on a time interval $[t_1, t_2]$ gives the Strichartz
estimates
\begin{equation}
\label{4.3.5}
\| v \|_{L^p L^q ( [t_1,t_2] \times \Rr^n)} \lesssim \| u_0 \|_{L^2}
\end{equation}
for ${2}/{p} + {n}/{q} = {n}/{2}$, $2 \leq p \leq \infty$,
$2 \leq q \leq 2 + {4}/({n-2})$ for $n \geq 3$ and $2 \leq q < \infty$
for $n =2$.  The Strichartz pair $p = q = {2 (n+2)}/{n}$ plays a special
role in scattering theory.
It is of vital importance that the pseudoconformal transform is also an
isometry on the Strichartz
spaces, namely we have
\begin{eqnarray}
\label{4.3.6}
\| v \|_{L^p L^q ( [-t_1^{-1}, -t_2^{-1}] \times \Rr^n)} = \| \tilde v
\|_{L^p L^q ( [t_1,t_2] \times \Rr^n)}.
\end{eqnarray}

\subsubsection{$L^r(\RR^n)$ decay rates and bounds in $\Sigma_\pi$} 

We want to prove $L^r$ decay estimates in time for solutions of
\eqref{eqn:NLSmc} parallel to Theorem $7.3.1$ in \cite{Caz_book},
but here in the case of the focusing,
mass critical nonlinearity.  We have the following result.

\begin{proposition} \label{p4.3.3}
Let $v$ be a solution of \eqref{eqn:NLSmc}, with initial data $v_0\in
\Sigma_\pi$, satisfying $\|v_0\|_{L^2}<\|Q_\pi\|_{L^2}$.
Let $u$ be defined as in \eqref{eqn:uvpct}.
For $2 \leq r \leq {2n}/({n-2})$ for $n \geq 3 $
and $2 \leq r < \infty$ for $n \leq 2  $, we have
\begin{equation}
\| u(t) \|_{L^r (\Rr^n)} = \|v(t)\|_{L^r(\RR^n)}
\leq C \langle t \rangle^{-n (1/2 - 1/{r})},
\label{4.3.10}
\end{equation}
for all $t \in \Rr$.
\end{proposition}

\begin{proof}
Note that $u(t)\in\Sigma_\pi$ for each $t$.  Hence, parallel to
\eqref{eqn:NLSmcgradeqE},
we have (with $p=1+4/n$)
\beq
\aligned
\frac{1}{2}\|\nabla u(t)\|_{L^2}^2
&=E(u(t))+\frac{1}{p+1}\|u(t)\|^{p+1}_{L^{p+1}} \\
&\le E(u(t))+\frac{C^{p+1}_{\Cal{H}}}{p+1}\|u(t)\|_{L^2}^{p-1}
\|\nabla u(t)\|^2_{L^2},
\endaligned
\eeq
and $\|u(t)\|_{L^2}=\|v(t)\|_{L^2}=\|v_0\|_{L^2}$, so, under our current
hypotheses, we have, parallel to \eqref{eqn:NLSmcgradbdE},
\beq
\|\nabla u(t)\|^2_{L^2}\le C E(u(t)),
\eeq
with $C<\infty$.  Then the Gagliardo-Nirenberg inequality plus the
conservation law \eqref{eqn:pctd} give
\beq
\aligned
\|v(t)\|_{L^r}=\|u(t)\|_{L^r}
&\le C_r \|\nabla u(t)\|_{L^2}^{n(1/2-1/r)}
\|u(t)\|_{L^2}^{1-n(1/2-1/r)} \\
&\le C'_r E(u(t))^{n(1/2-1/r)/2} \|u(t)\|_{L^2}^{1-n(1/2-r/2)} \\
&= C''_r t^{-n(1/2-1/r)}\|v_0\|_{L^2}^{1-n(1/2-1/r)},
\endaligned
\eeq
as desired.
\end{proof}

We next establish bounds in $\Sigma_\pi$.

\begin{proposition} \label{p4.3.4}
For $v$ as in Proposition \ref{p4.3.3}, we have the $\Sigma_\pi$ bound
\beq
\|e^{-it\Delta} v(t)\|_{H^1}+\|x e^{-it\Delta}v(t)\|_{L^2}\le C<\infty,
\eeq
for all $t\in\RR$.
\end{proposition}

\begin{proof}
The $H^1$-bound on $v(t)$ has been discussed, and $\{e^{-it\Delta}:t\in\RR\}$
is uniformly bounded on $H^1(\RR^n)$.  By \eqref{4.3.B}, $xv(t)$ is not
bounded in $L^2(\RR^n)$.  However, \eqref{eqn:pctd} says
\beq
\|x e^{-it\Delta}v(t)\|^2_{L^2}=\|xv_0\|_{L^2}^2
+\frac{t^2}{2+4/n}\|v(t)\|_{L^r}^r,\quad r=2+\frac{4}{n},
\eeq
and taking $r=2+4/n$ in \eqref{4.3.10} gives
$$
rn\Bigl(\frac{1}{2}-\frac{1}{r}\Bigr)=2,
$$
and hence provides the desired bound.
\end{proof}

\subsubsection{Proof of Proposition \ref{t:scattering}}

We will examine the behavior of $e^{-it\Delta}v(t)$ as $t\rightarrow +\infty$.
The behavior as $t\rightarrow -\infty$ has a parallel treatment.

Under the hypotheses of Proposition \ref{t:scattering}, $v$ solves
\eqref{eqn:NLSmc} for $t>0$ and $\tilde{v}$ solves \eqref{eqn:NLSmc} for $t<0$.
Furthermore, $\tilde{v}(t)\in \Sigma_\pi$ for each $t<0$, and $\|\tilde{v}(t)
\|_{L^2}=\|v(-1/t)\|_{L^2}=\|v_0\|_{L^2}<\|Q_\pi\|_{L^2}$.  Thus $\tilde{v}$
continues past $t=0$ as a global solution to \eqref{eqn:NLSmc}.  It follows that
\beq
\|\tilde{v}\|_{L^qL^r([-1,0]\times\RR^n)}<\infty,
\label{4.3.19}
\eeq
and hence that
\beq
\|v\|_{L^qL^r([1,\infty)\times\RR^n)}<\infty,
\label{4.3.20}
\eeq
for each Strichartz-admissible pair $(q,r)$, in particular for
\beq
q=r=\frac{2(n+2)}{n}.
\label{4.3.21}
\eeq
Now the solution $v(t)$ satisfies the integral equation
\beq
v(t)=e^{it\Delta}v_0+i\int_0^t e^{i(t-s)\Delta}\varphi(v(s))\, ds,
\label{4.3.22}
\eeq
with
\beq
\varphi(v)=|v|^{4/n}v.
\label{4.3.23}
\eeq
Hence, for $0\le t_1<t_2<\infty$,
\beq
\aligned
e^{-it_2\Delta}v(t_2)-e^{-it_1\Delta}v(t_1)
&=i\int_{t_1}^{t_2} e^{-is\Delta}\varphi(v(s))\, ds \\
&=T^* H_{t_1 t_2},
\endaligned
\label{4.3.24}
\eeq
where
\beq
H_{t_1,t_2}(s,x)=i\chi_{[t_1,t_2]}(s) \varphi(v(s,x)),
\label{4.3.25}
\eeq
and
\beq
T^*H(x)=\int_{-\infty}^\infty e^{-is\Delta} H(s,x)\, ds
\label{4.3.26}
\eeq
defines $T^*$ as the adjoint of $T$, given by
\beq
Tv_0(t,x)=e^{it\Delta}v_0(x).
\label{4.3.27}
\eeq
Strichartz estimates yield $T:L^2(\RR^n)\rightarrow L^qL^r(\RR\times\RR^n)$
and
\beq
T^*:L^{q'}(\RR,L^{r'}(\RR^n))\longrightarrow L^2(\RR^n),
\label{4.3.28}
\eeq
for Strichartz admissible pairs, including \eqref{4.3.21}.  Hence
\beq
\|e^{-it_2\Delta}v(t_2)-e^{-it_1\Delta}v(t_1)\|_{L^2}
\le C\|\varphi(v)\|_{L^{q'}([t_1,t_2]\times\RR^n)},
\label{4.3.29}
\eeq
with
\beq
q=\frac{2(n+2)}{n},\quad q'=\frac{2(n+2)}{n+4},
\label{4.3.30}
\eeq
hence
\beq
\aligned
\|\varphi(v)\|_{L^{q'}(I\times\RR^n)}^{q'}
&=\int\limits_I \int\limits_{\RR^n}
|v(s,x)|^{q'(1+4/n)}\, dx\, ds \\
&=\|v\|^q_{L^q(I\times\RR^n)}.
\endaligned
\label{4.3.31}
\eeq
Thus \eqref{4.3.20} guarantees that
\beq
\|e^{-it_2\Delta}v(t_2)-e^{-it_1\Delta}v(t_1)\|_{L^2}\longrightarrow 0,\quad
\text{as }\ t_1,t_2\rightarrow +\infty.
\label{4.3.32}
\eeq
We hence have a limit
\beq
e^{-it\Delta}v(t)\longrightarrow v_+
\label{4.3.33}
\eeq
in $L^2(\RR^n)$.  Then the uniform bounds in $\Sigma_\pi$ given in
Proposition \ref{p4.3.4} imply that $v_+\in \Sigma_\pi$, and convergence in
\eqref{4.3.33} holds $\text{weak}^*$ in $\Sigma_\pi$.
Further arguments, parallel to those on pp.~220--221 of \cite{Caz_book},
yield norm convergence.


\begin{remark}
\label{rem:altscat}
We note that scattering arguments for \eqref{eqn:NLSmc} also exist in spaces
with fewer regularity restrictions in $\Rr^n$.
In particular, there is the following proposition from \cite{BlCo}.
Here, $Q_1$ denotes the Weinstein functional maximizer $Q_\pi$ when $\pi=1$
is the trivial representation of $SO(n)$ on $\CC$.  We say
$v_0\in H^{0,s}(\RR^n)$ if and only if $\langle x\rangle^s v_0\in L^2(\RR^n)$,
and set
$$
\|v_0\|_{H^{0,s}}=\|\langle x\rangle^s v_0\|_{L^2}.
$$

\begin{proposition}[Blue-Colliander]
\label{t:blco}
Let $s\in (0,1]$.
Assume \eqref{eqn:NLSmc} is globally well posed on that subset of
$u_0\in H^s(\RR^n)$ satisfying $\|u_0\|_{L^2}<\|Q_1\|_{L^2}$.
Then, if $v_0\in H^{0,s}(\RR^n)$ satisfies $\|v_0\|_{L^2}<\|Q_1\|_{L^2}$,
there is a solution $v(t,x)$ to \eqref{eqn:NLSmc} for all $t$.
Furthermore, there are $v_{\pm}\in H^{0,s}(\RR^n)$ such that
$$
\lim\limits_{t\rightarrow\pm\infty}\ \|e^{-it\Delta}v(t)-v_{\pm}\|_{H^{0,s}}
=0.
$$
\end{proposition}

In addition, both \cite{GiVe} and \cite{Tsutsumi} give an early treatment of
scattering
theory in weighted spaces for repulsive nonlinearities.  In more recent
developments, concentration
compactness and frequency growth bounds via a refined interaction Morawetz estimate have been applied in the work of
Visan et al in \cite{KiTaVi,KiViZh0,TaViZh} and in the very recent
works of Dodson \cite{Dod,Dod1} to prove scattering purely in the space $L^2$
to handle both cases of defocusing mass critical NLS as well as the
focusing mass critical NLS with initial data mass below the ground state.
One is certainly tempted to speculate that such techniques can be applied to the
present vortex situation, but we do not pursue this here.
\end{remark}

\subsubsection{Existence of wave operators}

Here we show that material in Chapter 7 of \cite{Caz_book} can also be
adapted to prove the following result on existence of wave operators.

\begin{proposition} \label{p4.3.E}
Take $v_+\in\Sigma_\pi$ such that $\|v_+\|_{L^2}
<\|Q_\pi\|_{L^2}$.  Then there exists a unique solution
$v\in C(\RR,\Sigma_\pi)$ to \eqref{eqn:NLSmc}, satisfying \eqref{4.3.X} below, such
that
\beq
\lim\limits_{t\rightarrow+\infty}\, e^{-it\Delta}v(t)=v_+.
\eeq
There is a corresponding result for $t\rightarrow -\infty$.
\end{proposition}

\begin{proof}
The argument on pp.~221--223 of \cite{Caz_book} yields $S<\infty$ and $v(t)$,
defined initially for $t\in [S,\infty)=I$, satisfying

\beq
\aligned
v&\in L^q(I,H^{1,q}(\RR^n)), \\
(x+it\nabla)v&\in L^q(I\times\RR^n), \\
\|v(t)\|_{L^q}&\le C|t|^{-2/q},
\endaligned
\label{4.3.X}
\eeq
with $q=2(n+2)/n$, as in \eqref{4.3.30}, such that
\beq
v(t)=e^{it\Delta}v_+ +i\int_t^\infty e^{i(t-s)\Delta}\varphi(v(s))\, ds,
\quad t\ge S,
\eeq
with $\varphi(v)=|v|^{4/n}v$, as in \eqref{4.3.23}.
The solution is produced as the fixed point of a contraction mapping,
and is unique.
This argument works equally well for the focusing equation
\eqref{eqn:NLSmc} as for its defocusing counterpart, which
was the main focus of Theorem 7.4.4 of \cite{Caz_book}.  It does not
require our hypothesized bound on $\|v_+\|_{L^2}$.  As further noted in
\cite{Caz_book}, such $v$ has the properties
\beq
v\in C(I,H^1(\RR^n)),\quad (x+it\nabla)v\in C(I,L^2(\RR^n)),
\eeq
hence $v(t)\in H^1(\RR^n)\cap H^{0,1}(\RR^n)$ for each $t\ge S$.  Symmetry
considerations guarantee that, in our setting, where $v_+\in\Sigma_\pi$,
\beq
v(t)\in\Sigma_\pi,\quad \forall\, t\ge S.
\eeq
Furthermore, as observed in \cite{Caz_book}, one has
\beq
\|e^{-it\Delta}v(t)-v_+\|_{H^1\cap H^{0,1}}\longrightarrow 0,\quad \text{as }\
t\rightarrow +\infty.
\eeq
A fortiori, convergence holds in $L^2(\RR^n)$, and conservation of mass
implies
\beq
\|v(t)\|_{L^2}=\|v_+\|_{L^2},\quad \forall\, t\ge S.
\eeq
Given that $\|v_+\|_{L^2}<\|Q_\pi\|_{L^2}$, the global existence results
of \S{\ref{s4.1}} finish off the proof of Proposition \ref{p4.3.E}.
Uniqueness for $v$ for all $t$ follows from its uniqueness for $t\ge S$ large,
via well posedness of the backward Cauchy problem.
\end{proof}

\appendix
\section{Auxiliary results}
\label{append}

As advertised in the Introduction, we have three appendices.  The first
gives some concrete criteria for the nonvanishing of $V_0$
(cf.~\eqref{eqn:V0req}), hence for $H^1_\pi(M)\neq 0$.  The second takes
an explicit look at $V_0$ for some important representations of $SO(4),
SU(2)$, and $U(2)$.  The third introduces a more general geometrical setting,
with an eye to unifying the work on axial vortices in Section
\ref{sec:vortaxial} of this paper and that on weakly homogeneous spaces in
\cite{CMMT}.

 \subsection{Criterion that $H^1_\pi (M) \neq 0$}
 \label{app:H1pineq0}

As we have seen in Section \ref{sec:vortintro}, the space $H^1_\pi (M)$,
defined by \eqref{eqn:Hspi} is nonzero, for the class of spaces $M$ and
groups $G \subset SO (n)$ considered in Sections
\ref{sec:vortspher}-\ref{sec:vortaxial}, if and only if
\begin{equation}
\label{eqn:V0req}
V_0 = \{ \phi \in V \ : \ \pi (k) \phi = \phi, \ \forall k \in K \} \neq 0,
\end{equation}
with $K$ defined as in \eqref{eqn:Krep1}.  Here, we consider when this holds,
in a somewhat more general setting.

Let $G$ be a compact Lie group, $K \subset G$ a closed subgroup.  Then
$X = G/K$ has a Riemannian metric for which $G$ is a group of
isometries, acting transitively on $X$, and there is a $p_0 \in X$ such that
$K$ is the subgroup of $G$ fixing $p_0$.  In Section \ref{sec:vortintro}, $X$
is diffeomorphic to $\SS^{n-1}$, but we do not require this here.  We want to
find unitary representations $\pi$ of $G$ on a finite dimensional inner
product space $V$ with the property \eqref{eqn:V0req}.  Note that if $\pi$ is
reducible and $P_j$ are orthogonal projections onto irreducible components
$V_j$ ($\sum P_j = I$), then $\phi \in V_0$ implies $\phi_j = P_j \phi$
satisfies $\pi (k) \phi_j = \phi_j$, $\forall\, k \in K$, for each $j$, so it
suffices to consider irreducible representations.

To formulate our criterion, it is useful to bring in the regular
representation $R$ of $G$ on $L^2 (X)$, defined by
\begin{equation}
\label{eqn:regrepdef}
R(g) f(x) = f( g^{-1} x).
\end{equation}
This splits into an infinite sequence of irreducible unitary representations
$\rho_j$ of $G$ on $V_j \subset C^\infty (X) \subset L^2 (X)$ (each finite
dimensional).  If $G = SO (n)$, $K = SO (n-1)$, these are all mutually
inequivalent, but that is not always the case.  The following characterizes
which irreducible representations have the property \eqref{eqn:V0req}
(compare to for instance \cite{Zelo-book}, page $80$).

\begin{proposition}
\label{prop:app1}
An irreducible unitary representation $\pi$ of $G$ on $V$ satisfies
\eqref{eqn:V0req} if and only if it is equivalent to one of the following
unitary representations $\rho_j$ described above, i.e., $\pi$ is contained in
$L^2 (X)$.
\end{proposition}

\begin{proof}
First, we note that each representation $\rho_j$ of $G$ on $V_j$ has the
property \eqref{eqn:V0req}.  Take a nonzero $\psi \in V_j$.  Pick $p_1 \in X$
such that $\psi (p_1) \neq 0$.  Then take $g_1 \in G$ such that $g_1 p_1 =
p_0$ and set $\tilde{\psi} = \rho_j (g_1) \psi$, i.e., $\tilde{\psi} (x) =
\psi (g_1^{-1} x)$, so $\tilde{\psi} (p_0) \neq 0$.  Now take $\phi = \int_K
\rho_j (k) \tilde{\psi}\, dk$.  We have
\begin{equation}
\label{eqn:phiapp1}
\phi (p_0) = \tilde{\psi} (p_0) \neq 0, \quad \rho_j (k) \phi = \phi,
\quad \forall\, k \in K.
\end{equation}

For the converse, if $\pi$ is an irreducible unitary representation on $V$
and \eqref{eqn:V0req} holds, then $\pi$ is equivalent to the restriction of
$R$ to the linear subspace of $C^\infty (X)$ that is the range of
\begin{equation}
\label{eqn:Phiapp1}
\Phi : V \to C^\infty (X) , \ \ \Phi (\psi) (g \cdot p_0) = ( \pi (g) \phi,
\psi).
\end{equation}
In fact, given $h \in G$, $\psi \in V$
\beq
\aligned
R(h) \Phi (\psi ) (g \cdot p_0) & = ( \pi ( h^{-1} g ) \phi, \psi )  \\
& = (\pi (g) \phi, \pi (h) \psi) \\
& = \Phi (\pi(h) \psi) (g \cdot p_0).
\endaligned
\label{eqn:ReqPhiapp1}
\eeq
\end{proof}

\begin{remark}
Using the Weyl orthogonality relations, \cite{Zelo-book} shows that the
multiplicity with which an irreducible $(\pi, V)$ is contained in $L^2 (X)$
is equal to $\dim V_0$, where $V_0$ is as in \eqref{eqn:V0req}.  In fact, if
we set
\begin{equation}
\label{eqn:Psimapapp1}
\Psi_\pi : V_0 \otimes V \to L^2(X), \quad \Psi_\pi ( \phi \otimes \psi) (g
\cdot p_0) = (\pi (g) \phi, \psi),
\end{equation}
then the range $L_\pi$ of $\Psi_\pi$ is the subspace of $L^2 (X)$ on which
the regular representation $R$ acts like the copies of $\pi$, and $\Psi_\pi:
V_0 \otimes V \to L_\pi$ is an isomorphism.
\end{remark}

 \subsection{Examples of $(\pi, V)$ when $n = 4$}
 \label{app:4d}
 
 Here we record examples of irreducible representations $\pi$ of $G$ on $V$,
 for which
 \begin{equation}
 \label{eqn:V0app2}
 V_0 = \{ \phi \in V \ : \ \pi (k) \phi = \phi, \ \forall k \in K \} \neq 0,
 \end{equation}
 in the following cases:
 \begin{align}
 G & = SO (4),  \quad K = SO (3), \label{eqn:app2case1} \\
 G & = SU (2),  \quad K = \{ I \} , \label{eqn:app2case2}  \\
  G & = U (2),  \quad K = U (1). \label{eqn:app2case3} 
  \end{align} 
 In all cases, $G$ acts transitively on $\SS^3$, the unit sphere in $\RR^4$.
 Also, in all cases, these representations will be contained in the regular
 representation of $G$ on $L^2 (\SS^3)$, and in fact $L^2 (\SS^3)$ breaks into
 direct sums of such representations (sometimes with multiplicity).
 
 As we have seen, the eigenspaces of the Laplace operator on $L^2 (\SS^3)$
 are irreducible for the $SO(4)$ action, and in each case, $\dim V_0 = 1$.
 We give an alternative description of these representations.  For this, it
 is convenient to note that $\SS^3$ is isometric to $SU(2)$, with a natural
 bi-invariant metric tensor.  In particular, $SU(2) \times SU(2)$ acts on
 $SU(2)$ as a group of isometries, via
 \begin{equation}
 \label{eqn:SU2mapapp2}
 (g,h) \cdot x = gxh^{-1}, \quad g,h,x \in SU(2),
 \end{equation}
 and $(g,h)$ acts as the identity precisely for $g=h=I$ and $g=h=-I$, so we
 have a $2$-fold covering
 \begin{equation}
 \label{eqn:SU2SO4app2}
 SU (2) \times SU (2) \to SO (4).
 \end{equation}
 Irreducible representations of $SO(4)$ compose with \eqref{eqn:SU2SO4app2} to
 give irreducible representations of $SU(2) \times SU(2)$.  All the irreducible
 of $SU(2) \times SU(2)$ are of the form
 \begin{equation}
 \label{eqn:pitimespiapp2}
 \pi_{jk} (g,h) = \pi_j (g) \otimes \pi_k (h),
 \end{equation}
 where $\{ \pi_j \}$ are the irreducible representations of $SU(2)$.  It is the
 standard to classify the irreducible representations of $SU (2)$ as
 \begin{equation}
 \label{eqn:dj2app2}
 D_{j/2}, \ \text{acting on} \ \CC^{j+1}, \ j \in \{ 0,1,2,\dots \},
 \end{equation}
 so the irreducible representations of $SU (2) \times SU (2)$ are
 \begin{equation}
 \label{eqn:pidtimesdapp2}
 \pi_{jk} (g,h) = D_{j/2} (g) \otimes D_{k/2} (h) , \ \text{on} \
 \mathcal{P}_{jk} = \CC^{j+1} \otimes \CC^{k+1}.
 \end{equation}
 Such a representation descends to $SO(4)$ if and only if $j/2-k/2$ is an
 integer, i.e. if and only if $j$ and $k$ have the same parity.
 
 Now, the injection $SO(3) \hookrightarrow SO(4)$ lifts to
 \begin{equation}
 \label{eqn:SO3app2}
 SU(2) \hookrightarrow SU(2) \times SU(2), \quad g \to (g,g).
 \end{equation}
 Hence, for $V = \mathcal{P}_{jk}$, we have $V_0 \neq 0$ if and only if the
 representation $D_{j/2} \otimes D_{k/2}$ of $SU(2)$ contains the trivial
 representation $D_0$.  Now, we have the Clebsch-Gordan series
 \begin{equation}
 \label{eqn:ClGoapp2}
 D_{j/2} \otimes D_{k/2} \approx D_{|j-k|/2} \oplus \cdots \oplus D_{(j+k)/2},
 \end{equation}
 so this contains $D_0$ if and only if $j = k$.  Hence, the irreducible
 representations of $SO(4)$ for which $V_0 \neq 0$ (hence, those occuring in
 $L^2 (\SS^3)$) are precisely
 \begin{equation}
 \label{eqn:pijjapp2}
 \pi_{jj} (g,h) = D_{j/2} (g) \otimes D_{j/2} (h), \ j \in \{ 0,1,2,\dots \},
 \end{equation}
 pushed from $SU(2) \times SU(2)$ down to $SO (4)$.  
 
 The regular representation of $SO(4)$ on the $j$-th eigenspace of $\Delta$
 (starting with $j=0$) is equivalent to $\pi_{jj}$.  Let us also recall that
 the $j$-th eigenspace is equal to the set of restrictions to $\SS^3$ of the
 space $\calH_j$ of harmonic polynomials on $\RR^4$, homogeneous of degree $j$.

 The injection $SU(2) \hookrightarrow SO(4)$ lifts to
 \begin{equation}
 \label{eqn:liftSU2app2}
 SU(2) \hookrightarrow SU(2) \times SU(2), \quad g \to (I,g).
 \end{equation}
 When the representation $\pi_{jj}$ of $SO(4)$ on $\mathcal{P}_{jj} \approx
 \calH_j$ is restricted to $SU(2)$, it breaks up into $j+1$ copies of the
 representation $D_{j/2}$ of $SU (2)$.  For $j = 0$, $\calH_0 \approx \CC$ and
 $\pi_{00}$ is trivial.  For $j = 1$, we have the following decomposition.
 Write the linear functions on $\RR^4$ as $x_1,y_1,x_2,y_2$ and set $z_1 =
 x_1 + i y_1$, $z_2 = x_2 + i y_2$.  Then,
 \begin{equation}
 \label{eqn:HplusHapp2}
 \calH_1 = \calH_{1,1} \oplus \calH_{1,-1},
 \end{equation}
 with
 \begin{align}
 \label{eqn:H11app2}
 \calH_{1,1} & = \Span \{ z_1, \ z_2 \}, \\
 \label{eqn:H1-1app2}
 \calH_{1,-1} & = \Span \{ \bar{z}_1, \ \bar{z}_2 \} .
 \end{align}
 
 For $j = 2$, we have
 \begin{equation}
 \label{eqn:H2app2}
 \calH_2 = \calH_{2,0} \oplus \calH_{2,2} \oplus H_{2,-2},
 \end{equation}
 with
  \begin{align}
 \label{eqn:H20app2}
 \calH_{2,0} & = \Span \{ |z_1|^2 - |z_2|^2, \ z_1 \bar{z}_2, \ \bar{z}_1
 z_2 \}, \\
 \label{eqn:H22app2}
 \calH_{2,2} & = \Span \{ z_1^2, \ z_2^2, \ z_1 z_2 \} , \\
  \label{eqn:H2-2app2}
 \calH_{2,-2} & = \Span \{ \bar{z}_1^2, \ \bar{z}_2^2, \ \bar{z}_1
 \bar{z}_2 \} .
  \end{align}
 The groups $SU(2)$ and $U(2)$ both act, irreducibly, on each space
 $\calH_{j,m}$ listed above.  The actions of $SU(2)$ on $\calH_{1,m}$
 ($m = 1, -1$) are equivalent, as are the actions of $SU(2)$ on $\calH_{2,m}$
 ($m = 2, 0 , -2$) and for $G = SU(2)$, $V = \calH_{j,m}$ $\Rightarrow V_0 =
 V$.
 
 On the other hand, the actions of $U(2)$ on the two spaces $\calH_{1,m}$ are
 not equivalent, nor are those of $U(2)$ on the three spaces $\calH_{2,m}$.
 In each case,
 \begin{equation}
 \label{eqn:pcalHapp2}
 p \in \calH_{j,m} \ \Longleftrightarrow \ p \in \calH_j \ \text{and} \
 p(e^{i \theta} z) = e^{ i m \theta} p(z).
 \end{equation}
 Also, for $G = U(2)$, $V = \calH_{j,m}$, the space $V_0$ is one dimensional,
 and is the span of the first element in the spanning set of $\calH_{j,m}$,
 as listed in \eqref{eqn:H11app2}--\eqref{eqn:H2-2app2},
 given the $U(1)$ action on $\RR^4 = \CC^2$ is
 \begin{equation}
 \label{eqn:CCrotapp2}
 (z_1,z_2) \to (z_1, e^{i \theta} z_2).
 \end{equation}
 
 The results \eqref{eqn:HplusHapp2} and \eqref{eqn:H2app2} are special cases of
 the following.
 
 \begin{proposition}
 \label{prop:app2}
 For $j \geq 1$,
 \begin{align}
 \label{eqn:calHdecompapp2}
 \calH_j & = \calH_{j,-j} \oplus \calH_{j,-j+2} \oplus \cdots \oplus
 \calH_{j,j} \\
 & = \bigoplus_{\ell = 0}^j \calH_{j,2 \ell - j}, \notag
 \end{align}
 with $\calH_{j,m}$ given by \eqref{eqn:pcalHapp2}.  The group $SU(2)$ acts on
 each factor $\calH_{j, 2 \ell - j}$ as $D_{j/2}$.
 For $V = \calH_{j, 2 \ell -j}$, $0 \leq \ell \leq j$, we have $\dim V_0 = 1$,
 if $G = U(2)$, and $V_0 = \calH_{j, 2 \ell -j }$ if $G = SU(2)$.
 \end{proposition}
 
 \begin{proof}

 All the claims follow from the observations above, provided we show that
 $\calH_{j, 2 \ell -j} \neq 0$ for each $\ell \in \{ 0, \dots, j \}$.  To show
 this, it is useful to see how the one-parameter group
 \begin{equation}
 \label{eqn:Etapp2}
 E(t) \subset U(2) \subset SO(4), \ \ E(t) z = e^{it } z,
 \end{equation}
 lifts to a one-parameter group $\tilde{E} (t)$ in $SU(2) \times SU(2)$, whose
 action on $\RR^4$ is given by left and right quaternionic multiplication:
\begin{equation}
\label{eqn:quatmultapp2}
(g,h) \cdot x = gxh^{-1}.
\end{equation}
Recall that the $SU(2)$ action is given by the right quaternionic
multiplication.  In this case, a calculation shows that
\begin{equation}
\label{eqn:tilEapp2}
\tilde{E} (t) = ( \mathfrak{e} (t) , I ), \quad \mathfrak{e} (t) =
\left( \begin{array}{cc}
\cos t & - \sin t \\
\sin t & \cos t 
\end{array} \right).
\end{equation}
Given that $SU(2) \times SU(2)$ acts on $\calH_j$ by \eqref{eqn:pijjapp2},
the fact that $\calH_j$ decomposes under the $\tilde{E} (t)$ action according
to \eqref{eqn:calHdecompapp2}, with each factor having dimension $j+1$,
follows from the standard results on the representation $D_{j/2}$ on
$\CC^{j+1}$.
\end{proof}

We remark that
\begin{equation}
\label{eqn:Hjjspanapp2}
\calH_{j,j} = \Span \{ z_1^a z_2^b \ : \ a,b \in \ZZ^+, \ a+b = j \}
\end{equation}
and
\begin{equation}
\label{eqn:Hj-mapp2}
\calH_{j,-m} = \{ \bar{f} \ : \ f \in \calH_{j,m} \}.
\end{equation}
If $V \in \calH_{j,j}$, then $V_0 = \Span \{ z_1^j \}$.  Generally, if
$0 \leq m \leq j$, a $U(1)$-invariant element of $\calH_{j,m}$ is a linear
combination of monomials
\begin{equation}
\label{eqn:pjmabapp2}
p_{jmab} (z) = z_1^m |z_1|^{2a} |z_2|^{2b}, \ a,b \in \ZZ^+, \ a+b =
\frac{ j-m}{2},
\end{equation}
with the property that such a linear combination is a harmonic polynomial.
A calculation yields
\begin{equation}
\label{eqn:lastapp2}
z_1^{j-2} \bigl( |z_1|^2 - (j-1) |z_2|^2 \bigr) \in \calH_{j,j-2},
\quad j \geq 2.
\end{equation}

 \subsection{Other geometrical settings}
 \label{app:etc}

Here we generalize the geometrical settings of Section \ref{sec:vortaxial},
in which we considered axial vortices.  Let $M$ be a smooth Riemannian
manifold, possibly with boundary, and two groups $G$ and $H$ both acting on
$M$ as groups of isometries.  We assume
\begin{equation}
\label{eqn:GHapp3}
G \ \text{is compact}, \ \text{and the actions of} \ G,H \ \text{commute}.
\end{equation}
We take a base point $q_0 \in M$ and assume the following.
\beq
\aligned
& \text{For each $R \in (0,\infty)$, there exists $K \in (0,\infty)$ such
that}  \\
& \text{if $\Omega \subset M$ is $G$-invariant and $\diam (\Omega) \leq R$}, \\
& \text{then there exists $h \in H$ such that $h \cdot \Omega \subset B_K
(q_0)$}.
\endaligned
\label{eqn:assumapp3}
\eeq
For $M = \RR^{n+k}$, we can take $G = SO (n)$ and $H \approx \RR^k$.  If $G$
is trivial, the hypothesis \eqref{eqn:assumapp3} is that $H$ makes $M$ a
``weakly homogeneous space'' as defined in \cite{CMMT}.  At the other extreme,
if $H$ is trivial, \eqref{eqn:assumapp3} says that for each $R \in (0,\infty)$,
there exists $K \in (0,\infty)$ such that if $\Omega \subset M$ is
$G$-invariant and $\diam (\Omega) \leq R$, then $\Omega \subset B_K (q_0)$.
Compare hypothesis \eqref{eqn:diamgrowth}.

Here is another family of examples satisfying
\eqref{eqn:GHapp3}--\eqref{eqn:assumapp3}.  Let $\HH^{n+1}$ be the
$(n+1)$-dimensional hyperbolic space, e.g. $\{ x \in \RR^{n+1} \ : \ x_{n+1}
> 0 \}$ with metric tensor
\begin{equation}
\label{eqn:hypmetapp3}
ds^2 = x_{n+1}^{-2} \sum_{k=1}^{n+1} dx_k^2.
\end{equation}
Take $G = SO(n)$ acting on $(x_1, \dots , x_n) \in \RR^n$ (or $G$ could be a
subgroup of $SO(n)$ acting transitively on $\SS^{n-1}$).  Then, we take $H$
to be the group of dilations
\begin{equation}
\label{eqn:dilateapp3}
\delta_r (x) = rx , \ r \in (0,\infty).
\end{equation}

We make one further hypothesis, satisfied by $\HH^{n+1}$ and by weakly
homogeneous spaces:
\begin{equation}
\label{eqn:Mbdapp3}
M \ \text{has bounded geometry}.
\end{equation}

Given a unitary reresentation $\pi$ of $G$ on a finite dimensional inner
product space $V$ we have the space $H^1_\pi (M)$ as defined in
\eqref{eqn:Hspi}.  We have the following criterion implying that
$H^1_\pi (M) \neq 0$.  Take $p_0 \in M$ (not necessarily the same as $q_0$).
Let
\begin{equation}
\label{eqn:Kdefapp3}
K = \{ k \in G \ : \ k p_0 = p_0 \},
\end{equation}
and 
\begin{equation}
\label{eqn:V0app3}
V_0 = \{ \phi \in V \ : \ \pi (k) \phi  = \phi, \ \forall k \in K \}.
\end{equation}
Then, assume
\begin{equation}
\label{eqn:piirredapp3}
\pi \ \text{irreducible}, \ \ V_0 \neq 0.
\end{equation}

\begin{lemma}
If \eqref{eqn:piirredapp3} holds, then $H^1_\pi (M) \neq 0$.
\end{lemma}

\begin{proof}
Given a nonzero $\phi \in V_0$, $v(g \cdot p_0) = \pi (g) \phi$ gives a
well-defined $C^\infty$ function (with values in $V$) on $\mathcal{O}_{p_0}$,
the $G$-orbit of $p_0$, which is a smooth, compact submanifold of $M$.
Extend this function to $v \in C_0^\infty (M,V)$.  Then, set
\begin{equation}
\label{eqn:udefapp3}
u(x) = (\dim V) \int_G v (g^{-1} x) \overline{ \Tr \pi (g) }\, dg.
\end{equation}
We have $u \in C_0^\infty (M) \cap H^1_\pi (M)$ and $u = v$ on
$\mathcal{O}_{p_0}$.
\end{proof}

Under the hypotheses \eqref{eqn:GHapp3}, \eqref{eqn:assumapp3} and
\eqref{eqn:Mbdapp3}, given $H^1_\pi (M) \neq 0$, one can use concentration
compactness arguments similar to those in Sections \ref{sec:vortspherenmin}
and \ref{sec:vortaxialFlambda} to establish the following.  If
\begin{equation}
\label{eqn:pbdsapp3}
1 < p < \frac{m+2}{m-2}, \quad m = \dim M,
\end{equation}
and
\begin{equation}
\label{eqn:specapp3}
\Spec (-\Delta) \subset [\delta, \infty),\quad \lambda > -\delta,
\end{equation}
then, given $\beta > 0$, $F_\lambda$ and $J_p$ as in
\eqref{eqn:Flambda}--\eqref{eqn:Jp}, one can find $u \in H^1_\pi (M)$
achieving
\begin{equation}
\label{eqn:IIapp3}
\II (\beta, \pi) = \inf \{ F_\lambda (u) \ : \ u \in H^1_\pi (M), \ J_p (u)
= \beta \}.
\end{equation}
On the other hand, if
\begin{equation}
\label{eqn:pbdsapp3-1}
1< p < 1 + \frac{4}{m},
\end{equation}
$\beta > 0$, and $E(u)$ and $Q(u)$ are as in \eqref{eqn:energy} and
\eqref{eqn:charge}, one can find $u \in H^1_\pi (M)$ achieving
\begin{equation}
\label{eqn:EEapp3}
\EE (\beta, \pi) = \inf \{ E(u) \ : \ u \in H^1_\pi (M), \ Q (u)  = \beta \}, 
\end{equation}
provided
\begin{equation}
\label{eqn:EEbdapp3}
\EE (\beta, \pi) < 0.
\end{equation}
Modifications of arguments of Sections \ref{sec:vortspherenmin}
and \ref{sec:vortaxialFlambda} to prove these results
are left to the reader.


\end{document}